\documentclass[12pt, a4paper]{amsart}   

\usepackage[paper=a4paper, left=3cm,right=2cm,top=3cm,bottom=3cm]{geometry}  %
\usepackage[utf8]{inputenc}
 
\usepackage{setspace}
\usepackage{epsfig}
\usepackage{bm}
\theoremstyle{plain}
\newtheorem{theo}{Theorem}
\newtheorem{lem}{Lemma}
\newtheorem{prop}[theo]{Proposition}
\newtheorem{cor}[theo]{Corollary}

\theoremstyle{definition}
\newtheorem{assumption}{Assumption}
\newtheorem{defi}{Definition}
\theoremstyle{remark}
\newtheorem{rem}{Remark}

\newcommand{\cal}{\mathcal}

\renewcommand{\H}{{\cal H}}

\newcommand{\EE}{{\cal E}}
\newcommand{\E}{{\mathbb{E}}}

\newcommand{\R}{{\mathbb{R}}}

\newcommand{\al}{\alpha}
\newcommand{\lam}{\lambda}

\newcommand{\es}{{\cal S}}

\newcommand{\h}{{\cal H}}

\newcommand{\X}{{\cal X}}

\newcommand{\prf}{\begin{proof}} 
\newcommand{\prfend}{\end{proof}} 

\newcommand{\la}{\langle}
\newcommand{\ra}{\rangle}
\newcommand{\glam}{g_{\lambda}}
\newcommand{\x}{{\bf x}}
\newcommand{\y}{{\bf y}}

\newcommand{\z}{{\bf z}}
\newcommand{\Z}{{\bf Z}}

\newcommand{\M}{{\cal M}}

\newcommand{\A}{{\cal A}}

\newcommand{\NN}{{\cal N}}

\newcommand{\eps}{\varepsilon}

\newcommand{\frho}{f_{\rho}}
\newcommand{\lamstar}{\lam_{*}}

\newcommand{\Lam}{\Lambda}

\newcommand{\bnl}{B_{n, \lam}}

\newcommand{\lipc}{\ell} 
\newcommand{\comfun}{\Delta}

\newcommand{\Xixl}{\Xi_{\x,\lam}}
\newcommand{\Tzl}{\Theta_{\z,\lam}}

\newcommand{\Gxl}{\Gamma_{\x,\lam}}
\newcommand{\bnlab}[2]{\bnl(#1,#2)}

\newcommand{\Hr}{\cal H}

\newcommand{\fo}{f_{\rho}}

\newcommand{\fzl}{f_{\z}^\lambda}
\newcommand{\fl}{f^\lambda}

\newcommand{\fd}{f_\rho}

\newcommand{\fls}{f^{\lamstar}}
\newcommand{\letan}{\wt{L}_{\eta,n}}
\newcommand{\leta}{L_{\eta}}

\newcommand{\hs}{{\mathrm{HS}}}
\newcommand{\tr}[1]{\operatorname{Tr}\left[#1\right]}

\newcommand{\wt}{\widetilde}
\newcommand{\wh}{\widehat}
\newcommand{\ol}{\overline}

\let\originalleft\left
\let\originalright\right
\renewcommand{\left}{\mathopen{}\mathclose\bgroup\originalleft}
\renewcommand{\right}{\aftergroup\egroup\originalright}

\RequirePackage{xstring}

\newcommand{\paren}[2][a]{%
\IfEqCase{#1}{%
{a}{\left(#2\right)}%
{0}{(#2)}%
{1}{\big(#2\big)}%
{2}{\Big(#2\Big)}%
{3}{\bigg(#2\bigg)}%
{4}{\Bigg(#2\Bigg)}%
}[\PackageError{paren}{Undefined option to paren: #1}{}]%
}
\newcommand{\norm}[2][a]{%
\IfEqCase{#1}{%
{a}{\left\lVert#2\right\rVert}%
{0}{\lVert#2\rVert}%
{1}{\big\lVert#2\big\rVert}%
{2}{\Big\lVert#2\Big\rVert}%
{3}{\bigg\lVert#2\bigg\rVert}%
{4}{\Bigg\lVert#2\Bigg\rVert}%
}[\PackageError{norm}{Undefined option to norm: #1}{}]%
}
\newcommand{\brac}[2][a]{%
\IfEqCase{#1}{%
{a}{\left[#2\right]}%
{0}{[#2]}%
{1}{\big[#2\big]}%
{2}{\Big[#2\Big]}%
{3}{\bigg[#2\bigg]}%
{4}{\Bigg[#2\Bigg]}%
}[\PackageError{brac}{Undefined option to brac: #1}{}]%
}
\newcommand{\inner}[2][a]{%
\IfEqCase{#1}{%
{a}{\left\langle#2\right\rangle}%
{0}{\langle#2\rangle}%
{1}{\big\langle#2\big\rangle}%
{2}{\Big\langle#2\Big\rangle}%
{3}{\bigg\langle#2\bigg\rangle}%
{4}{\Bigg\langle#2\Bigg\rangle}%
}[\PackageError{inner}{Undefined option to inner: #1}{}]%
}
\newcommand{\abs}[2][a]{
\IfEqCase{#1}{%
{a}{\left\vert#2\right\rvert}%
{0}{\vert#2\rvert}%
{1}{\big\vert#2\big\rvert}%
{2}{\Big\vert#2\Big\rvert}%
{3}{\bigg\vert#2\bigg\rvert}%
{4}{\Bigg\vert#2\Bigg\rvert}%
}[\PackageError{abs}{Undefined option to abs: #1}{}]%
}

\newcommand{\set}[2][a]{
\IfEqCase{#1}{%
{a}{\left\{#2\right\}}%
{0}{\{#2\}}%
{1}{\big\{#2\big\}}%
{2}{\Big\{#2\Big\}}%
{3}{\bigg\{#2\bigg\}}%
{4}{\Bigg\{#2\Bigg\}}%
}[\PackageError{set}{Undefined option to set: #1}{}]%
}

\newcommand{\scalar}[3]{\langle{ #1},{#2} \rangle_{#3}}

\RequirePackage{color}

\newcommand{\pmathe}[1]{\textcolor{black}{#1}}
\newcommand{\gillesb}[1]{\textcolor{black}{#1}}
\newcounter{nbdrafts}
\makeatletter
\newcommand{\checknbdrafts}{
\ifnum \thenbdrafts > 0
\@latex@warning@no@line{**********************************************************************}
\@latex@warning@no@line{* The document contains \thenbdrafts \space draft note(s)}
\@latex@warning@no@line{**********************************************************************}
\fi}

\makeatother

\graphicspath{{./Images/}}

\newcommand{\incl}{T}

\numberwithin{equation}{section}
\numberwithin{theo}{section}

 \title{Lepski\u\i\/ principle in supervised learning}
\author{Gilles Blanchard}
\address{Institut f\"ur Mathematik, Universit\"at Potsdam,
Karl-Liebknecht-Stra{\ss}e 24-25 14476 Potsdam, Germany}
\email{gilles.blanchard@math.uni-potsdam.de}
\author{Peter Math\'e}
\address{Weierstrass Institute, Mohrenstrasse 39, D-10117 Berlin,
  Germany}
\email{peter.mathe@wias-berlin.de}
\author{Nicole M\"ucke}
\address{Institute for Stochastics and Applications, 
Faculty 8: Mathematics and Physics, 
University of Stuttgart, 
D-70569 Stuttgart Germany }
\email{nicole.muecke@mathematik.uni-stuttgart.de}
\date{Version: \today}
\begin{document}
\begin{abstract}
  In the setting of supervised learning using reproducing kernel methods,
  we propose a data-dependent regularization parameter selection rule that is
  adaptive to the unknown regularity of the target function and is
  optimal {\em both} for the least-square (prediction) error and for the
  reproducing kernel Hilbert space (reconstruction) norm error.
  It is based on a modified Lepski\u\i\/ balancing principle using a varying family of norms.
\end{abstract}
 \maketitle

 \section{Introduction}
 \label{sec:introduction}

 We shall study optimal reconstruction of the regression function in
 supervised learning. Here we are given observations
 \begin{equation}
   \label{eq:base}
   Y_{i} := f_\rho(X_{i}) + \varepsilon_{i},\qquad i=1,\dots,n,
 \end{equation}
 at i.i.d. data points~$X_{1},\dots,X_{n}$, drawn according to some
 (unknown) distribution~$\rho_X$ on a space~$\X$. If~$\hat f_{z}$ is any predictor for
 the regression function~$f$ based on a
 sample~$\z:= \paren{X_{i},Y_{i}}_{i=1}^{n}$\;, then we measure the
 (squared) loss (also called excess squared prediction risk) as
 \begin{equation}
   \label{eq:loss}
   \EE(f,f_{z}) = \E_{X\sim \rho_X} \abs{\paren{f_\rho(X) - \hat f_{\z}(X)}}^{2} = \norm{f_\rho - \hat f_{\z}}_{L^{2}(\rho_X)}^{2}.
 \end{equation}
 We shall adopt a framework commonly considered in learning theory for so-called
 {\em reproducing kernel} methods, and assume that both the target $f_\rho$ and its
 estimation $\hat f_\z$
 belong to a given reproducing kernel Hilbert space (rkhs) $\Hr$. In this setting it is also
 relevant to study the estimation loss in $\Hr$-norm,
 $ \norm[1]{f_\rho - \hat f_\z}_\Hr$. In particular, one application of interest is when $f_\rho=Ah_\rho$
 with $h_\rho$ an element of a Hilbert space $\Hr_0$, $A$ is a (known) linear operator from $\Hr_0$ to the set of real functions $\X \rightarrow \R$, and one is interested in reconstructing $h_\rho$ with small $\Hr_0$-norm error. This setting known as {\em inverse regression} can be shown to be formally equivalent
 to the rkhs setting provided $A$ satisfies some regularity properties (namely continuity
 of all evaluation functionals $h\mapsto (Ah)(x)$), see~\cite{BlaMuc18,discretization}. In this type of application, it is of interest to get an optimal control of
 the direct (or prediction) error $\norm[1]{A(h_\rho - \hat h_\z)}_{L^2(\rho_X)}$, as well as of the reconstruction error $\norm[1]{h_\rho - \hat h_\z}_{\Hr_0}$; the latter error coincides
 with $\norm[1]{f_\rho - \hat f_\z}_{\Hr}$ for a suitable rkhs $\Hr$ depending on $A$ and isometric to $\Hr_0$~\cite{BlaMuc18}.

 Specifically, we shall confine the analysis to estimators constructed
 from a \emph{linear regularization scheme}~$\glam$ as
 \begin{equation}
   \label{eq:fzl}
   \fzl := \glam(B_{\x}) \incl_{\x}^{\ast}\y,\qquad 0< \lambda \leq \kappa^2,
 \end{equation}
\pmathe{ with~$\kappa$ specified in~\S~\ref{sec:defops}.}

 where $\y=(y_1,\ldots,y_n) \in \R^n$, and $T^*_\x$, $B_\x$ are data-dependent operators
 from $\R^n$, resp $\Hr$, to $\Hr$, 
 and $g_\lam$ is a regularization function; these objects will
 be more precisely defined in the next section. The regularization parameter $\lambda$
 will drive the bias/variance trade-off, with smaller values of $\lambda$ corresponding to less regularization, that is, smaller bias and larger variance.

 The above setting has been analyzed in numerous previous studies, and a literature overview
 will be necessarily partial. Let us mention the seminal works~\cite{CaponnettoDevito2007,discretization,SmaleZhou2007,SteHusSco09,mendelson2010} which focused (mainly) on
 optimal bounds for the prediction risk, for Tikhonov regularization schemes. Relationship
 of the model~\eqref{eq:base}, viewed through the lens of reproducing kernel methods, to the
 inverse problem literature was pioneered in~\cite{learning,Gerfo,discretization}, opening the way to using more general  regularization schemes of the form~\eqref{eq:fzl}. Statistical performance
 bounds covering such schemes were established in~\cite{CaponnettoYao2010} for prediction risk,
 \cite{BPR2007} both for prediction and reconstruction risk (albeit in a worst-case setup concerning
 the spectral decay of the kernel integral operator, giving rise to so-called ``slow rates''),
 and \cite{BlaMuc18} concerning fast rates for both risks.

 All of these studies (with the exception of~\cite{CaponnettoYao2010} which considered
 data-dependent parameter selection, see below) studied convergence rates under
 {\em a priori} known regularity assumptions on the target function $f_\rho$ (expressed under
 the form of specific {\em source conditions}, which will be defined below) and, for
 fast (minimax optimal) rates, additional {\em a priori} known assumptions on the spectral decay of the the kernel
 operator. The optimal choice of the regularization parameter $\lambda$ depends of these conditions;
 however in most practical applications, such regularity assumptions are unknown to the user, and
 it is a fundamental task to select the regularization parameter~$\lambda$ in a close to optimal
 way from the data only; this is the focus of the present paper.
 
 Several strategies are known for this.  Concerning the prediction risk,
 a standard approach is hold-out or cross-validation (considered in~\cite{CaponnettoYao2010}),
 which picks among a set of estimators $(\fzl)_{\lam \in \Lambda}$ (with $\Lambda$ a finite set, typically a geometric discretization of a certain parameter range) the one having smallest empirical error when evaluated on an independent, "hold-out" sample of the same size as the original. It is an established fact that this method is able to attain close to optimal rates
 in a variety of situations, since it satisfies an oracle-type inequality with respect to the
 considered estimator family, at least if it is assumed that $Y$ is bounded \cite{CaponnettoYao2010,BlaMas06}. It is crucial for this that the empirical excess risk is an accessible unbiased estimate of the population excess risk. Therefore, such a method will not be applicable for the $\Hr$-norm risk (reconstruction error). When assuming smoothness properties under the form of source conditions for $f_\rho$, it is known that there exists a choice of the regularization parameter which is simultaneously minimax optimal over that smoothness class for the direct and for the reconstruction risk. One might therefore hope that
 the hold-out method parameter selection will return a parameter choice that also features optimality properties for the reconstruction risk: however, existence of one regularization parameter that is
 good for both risks (that is to say, is minimax rate optimal over a certain regularity class) is no guarantee that
 {\em any arbitrary} parameter which is good for prediction (e.g. as selected by the hold-out
 method), is also good for reconstruction: there could in principle be a range of parameters that are
 good for prediction, but only one element in that range that is also good for reconstruction,
 without guarantee that precisely that one will be selected by hold-out. 
 We must therefore consider possible other approaches.
 
 A prominent parameter choice strategy which uses only differences of
 estimators is based on Lepski\u\i's original idea~\cite{Lepski90}. This idea can be  adopted in learning.
 Concerning the prediction risk ($L^2(\rho_X)$-norm), this was studied
 in~\cite{DeVito2010} (obtaining ``slow rates''); and
 further in the  monograph~\cite{LP2013}, and the paper~\cite{LuMaPer2018}, where fast prediction rates where obtained under some 'minimal smoothness' assumptions. 
 Lepski\u\i's principle was recently used in learning within the context of empirical risk minimization over balls in the reproducing kernel Hilbert space. Here the size of the ball, which reflects the inherent solution smoothness, is chosen adaptively, see~\cite{SteGr2018} for details. This approach is not able to take into account additional properties of the marginal distribution~$\rho_X$, as these are quantified in the effective dimension (in other words, once again only "slow rates" are considered in that work).
 We mention that Lepski\u\i's principle can
 also be used for adaptation in the reconstruction risk ($\Hr$-norm), but, as discussed above, any method concentrating only on adaptation for the one risk might be sub-optimal for the other.

 The contributions of the present study are the following.
 We derive a data-dependent regularization parameter selection rule, based on a
 modified Lepski\u\i's principle, that is {\em simultaneously} adaptive in the sense of
 the prediction and the reconstruction risks. The simultaneous adaptation property
 is a new aspect of the method that has not appeared previously, up to our knowledge.
 This is achieved by using a selection rule based on a varying norm to measure
 difference of estimators, wherein the norm depends on the regularization parameter.
 We establish results showing simultaneous oracle-type inequalities for both risks, giving rise to
 fast (minimax optimal) convergence rates adapting to unknown source conditions of a general form,
 as well as to unknown spectral decay properties of the kernel integral operator.
 We insist that in contrast to some earlier studies,
 the rule is entirely data-dependent: in particular the variance of the
 estimators, as represented by the so-called {\em effective dimensionality} of each
 estimator, and depending on the (a priori unknown) spectral decay of the kernel
 integral operator, is also estimated empirically, as is the ``minimal'' regularization
 parameter, which also depends on said effective dimension. In comparison to~\cite{LuMaPer2018}
 for the prediction risk, we also remove the ``minimum smoothness'' requirement.
 Finally, we argue that the selection rule is simple to put in practice, as the varying
 squared norm used in the rule is a linear combination of the (squared) empirical norm and
 $\Hr$-norms of the estimators, both of which are readily accessible. The main ideas in this paper, especially concerning the extended Lepski\u\i's principle, were introduced under a preliminary form in the third author's Ph.D. thesis \cite{nicole_phd}.

 This paper is organized as follows: in Section~\ref{sec:notation}, we introduce the setting and
 notation for supervised learning using (reproducing) kernel-based estimators and general regularization schemes. In Section~\ref{sec:basic-results}, we present a first
 simultaneous adaptation result using the newly introduced modified Lepski\u\i's principle,
 in a somewhat abstracted setting allowing us to put into light the main ideas.
 In Section~\ref{sec:ass3}, we present the actual method and bounds in the supervised
 learning setting. The appendix contains technical proofs,
 including in particular probabilistic estimates for various error terms used as fundamental
 building blocks for the main results. While estimates of a similar flavor have appeared in
 various earlier studies on reproducing kernel learning methods, we have adopted here
 a self-contained
 approach, in passing streamlining or slightly extending earlier arguments. For instance, we
 introduce a technical device allowing us to let go of the assumption
 of operator monotonicity 
 in the generalized source condition function appearing in all earlier studies.


\section{Mathematical Framework in Supervised Learning}
\label{sec:notation}

\subsection{Operators}
\label{sec:defops}
We let $\Z = \X \times \R$ denote the sample space, where the input space $\X$ is a standard Borel space. 
The relation between the input $x \in \X$ and the output $y \in \R$ is
described by a \pmathe{unknown joint} probability distribution $\rho$ on $\X \times \R$. By $\rho_X$ we denote 
the $\X-$ marginal distribution. 
\\
Based on  a {\it training set} \;$\z=(\x, \y)=((x_1, y_1),...,(x_n,y_n))$ drawn independently and identically distributed according to $\rho$, the goal in supervised learning is to find a function $\hat f=\hat f_{\z}: \X \to \R$ with small expected error 
\[  \EE(\hat f)=\int_{\X \times\R} (y - \hat f(x))^2 \;d\rho(x,y) \;. \]
It is well known that the minimizer of $\EE(\cdot)$ over $L^2(\X, \rho_X)$ is the regression function, denoted by $f_{\rho}$, satisfying
\[  \| f- f_{\rho} \|^2_{\rho_X} = \EE(f)-\EE(f_{\rho}) \;, \qquad f \in L^2(\X, \rho_X)\;, \]
with $\|\cdot\|_{\rho_X}$ being the $L^2(\X,\rho)-$ norm.
\\
It is common to search for an estimator in a hypothesis space $\h \subset L^2(\X, \rho_X)$, which we choose to be a separable reproducing kernel Hilbert space (RKHS) $\h=\h_K$, 
arising from a  measurable positive semi-definite kernel $K: \X \times \X \longrightarrow \R$. For any $x\in \X$ we denote $K_x$ the element of $\h$
given by the function $t \mapsto K(x,t)$. We recall the fundamental ``reproducing'' property
$f(x) = \inner{K_x,f}_\h$ holding for any $f\in \H,x\in \X$.
Additionally, we let $K$ be bounded by 
$\kappa^2=\sup_{x\in \X}K(x,x)$. 
\\
Following previous studies~\cite{learning,BPR2007,BlaMuc18} we consider the continuous inclusion $\incl=\incl_{\rho_X}: \h \hookrightarrow  L^2(\X, \rho_X)$ and its 
adjoint $\incl^*=\incl_{\rho_X}^*: L^2(\X, \rho_X) \rightarrow \h$, leading 
to the kernel second moment operator $B={B_{\rho_X}}=\incl^*\incl: \h \rightarrow \h$,     
which can be shown to be positive, self-adjoint and trace class (and
hence in particular compact), see e.g.~\cite{learning}. 
Given a sample $\x=(x_1, \ldots, x_n) \in \X^n$, we define the empirical counterparts $\incl_{\x}: \h \rightarrow  \R^n$ by 
$(\incl_{\x}f)_i= f(x_i)=\la f,  K_{x_i} \ra_{\h}$ and, endowing $\R^n$ with the normalized scalar product
$\inner{u,v} = n^{-1} \sum_{i=1}^n u_i v_i$,  its adjoint~$\incl^*_{\x}: \R^n \rightarrow \h$
given by $\incl^*_{\x} u = n^{-1} \sum_{i=1}^n u_i K_{x_i}$.
The empirical 
second moment operator is then given by $B_{\x}= \incl_{\x}^*  \incl_{\x}$. Moreover, we have the relations $B
=\E_{X \sim \rho_X}[K_X \otimes K_X^*]$,
$B_{\x}=n^{-1} \sum_{i=1}^n K_{x_i} \otimes K_{x_i}^*$, \pmathe{and
that~$\norm{B_{\x}}_{\Hr\to \Hr},\norm{B}_{\Hr\to \Hr}\leq \kappa^{2}$. }
We refer to \cite{discretization,BlaMuc18} for more details.

We shall use the moment inequality in Hilbert space,
see~\cite[Chapt.~2.3]{EHN1996}, which asserts (for an arbitrary
non-negative self-adjoint operator~$B$) that for every~$0 \leq \theta
\leq 1$ it holds true that
\begin{equation}
  \label{eq:moment-ineq}    
  \|B^{\theta}f\|_{\Hr} \leq \|B f\|^{\theta}_{\Hr}
  \; \|f\|^{1- \theta}_{\Hr},\quad f\in\Hr,
\end{equation}

\subsection{Regularization}
\label{sec:regularization}

Regularization methods have emerged as a useful tool in Learning Theory for tackling the problem of overfitting and have been  
introduced in e.g. \cite{wahba}, \cite{poggio92}, 
\cite{hastie01}.
We confine ourselves to the class of {\it spectral regularization methods} $\{g_{\lam}\}_\lam$, examined in e.g. 
\cite{Gerfo,BPR2007,CaponnettoDevito2007,BlaMuc18} but also in \cite{bissantz} in a statistical setting. Here, 
$0< \lam \leq  \kappa$ denotes the regularization parameter. This class of methods contains the well known Tikhonov regularization, 
Landweber iteration or spectral cut-off. 
We recall its definition. 

\begin{defi}[\protect{\bf Regularization function}]
Let $g: (0,\kappa^2]\times [0, \kappa^2] \longrightarrow \R$ be a function and write 
$g_{\lam}=g(\lam, \cdot)$. The family $\{g_{\lam}\}_{\lam}$ is called 
{\it regularization function}, if there are positive constants $\gamma_0$ and $\gamma_{-1}$ such that for any $0 < \lam \leq \kappa^2$:
\begin{equation}
  \label{eq:supg}
\sup_{0<t\leq \kappa^2}|g_{\lam}(t)| \leq \frac{\gamma_{-1}}{\lam} \,,
\end{equation}
and
\begin{equation}
    \label{eq:resid}
\sup_{0<t\leq \kappa^2}|r_{\lam}(t)| \leq \gamma_0 \,,
\end{equation}
where $r_{\lam}(t):= 1-g_{\lam}(t)t$ is the residual function.
Observe that the latter property implies in particular
\begin{equation}
  \label{eq:supgt}
\sup_{0<t\leq \kappa^2}|tg_{\lam}(t)| \leq \gamma_0+1 \,.
\end{equation}
\end{defi}

Given a regularization function~$g_{\lam}$ we use spectral calculus to
apply this to the self-adjoint nonnegative operator~$B_{\x}$. Hence we consider the
family~$\fzl$ of (linear) estimators given by~\eqref{eq:fzl}, repeated here for convenience:
\[
\fzl := g_{\lam}(B_{\x}) T_{\x}^{\ast}\y,\quad 0 < \lam\leq \kappa^2,
\]
The above can be represented equivalently, using the ``shift'' formula,
as a kernel expansion
\begin{equation} \label{eq:shift}
  f_\z^\lam = n^{-1} \sum_{i=1}^n \al^\lam_i K_{x_i} = T^*_\x \bm{\al}^\lam, \text{ where }
  \bm{\al}^\lam := g_\lam(T_\x T_\x^*)\y;
\end{equation}
since $T_\x T_\x^*$ is a $n\times n$ matrix (whose $(i,j)$-entry is equal to $n^{-1}K(x_i,x_j)$),
computation of $\fzl$ boils down to numerically feasible matrix computations.

The objective of this study is to provide adaptivity results for
estimators~$\fzl$, where the final estimator is defined using an {\it
  a posteriori} parameter choice $\hat \lam = \hat \lam(n,\z)$ (see
Definition~\ref{def:lepest}), resulting in~$f_{\z}^{\hat\lam}$.

The quality of these estimators depends on the capability of the
regularization to take into account smoothness, and this is concerned
with the notion of a qualification. To this end we agree with the
following concept.
\begin{defi}[\protect{\bf Index function}] \label{def:index} A continuous nondecreasing
  function~$\phi\colon [0,\kappa^{2}]\to \R^{+}$ which obeys~$\phi(0)=0$  is called an index
  function. Index functions are endowed with the following partial order:
  for two index functions $\phi_1,\phi_2$ we have $\phi_1 \prec \phi_2$
  (we say that $\phi_2$ {\em covers} $\phi_1$) if
  $\phi_2/\phi_1$ is nondecreasing for $x>0$.  
\end{defi}
\begin{defi}[\protect{\bf Qualification}]
  \label{def:qualif}
We call an index function~$\psi$ a \emph{qualification} of the chosen regularization function~$g_\lam$ if there is a constant~$\gamma_\psi$ such that
\begin{equation}\label{ref:qualif}
\sup_{0<t\leq \kappa^2}|r_{\lam}(t)|\psi(t)\leq \gamma_\psi \psi(\lam),\quad \lam>0.
\end{equation}
\end{defi}

We mention a standard and useful consequence of the above definitions obtained by
interpolation.
\begin{prop}
  \label{prop:interpreg}
  Let $g_\lambda$ be a regularization function with qualification $\psi$.
  Let $\ol{\gamma}_\psi:= \max(\gamma_0+1,\gamma_{-1},\gamma_\psi)$. Then
  for any $\lambda\in(0,\kappa^2]$:
  \begin{itemize}
  \item For any $r \in [0,1]$:
    \begin{equation}
      \sup_{0<t\leq \kappa^2}|g_{\lam}(t)|t^r \leq
      \ol{\gamma}_\psi \lam^{r-1}.
    \end{equation}
  \item For any index function $\phi$ covered by $\psi$:
    \begin{equation}
      \sup_{0<t\leq \kappa^2}|r_{\lam}(t)|\phi(t) \leq
      \ol{\gamma}_\psi \phi(\lambda).
    \end{equation}
  \end{itemize}
\end{prop}
\begin{proof}
  For the first statement, we have using~\eqref{eq:supg},~\eqref{eq:supgt}, for any $(t,\lambda)\in(0,\kappa^2]\times[0,\kappa^2]$:
  \[
    \abs{g_\lambda(t)}t^r = \abs{g_\lambda(t)t}^r \abs{g_\lambda(t)}^{1-r}
    \leq (\gamma_0+1)^r(\gamma_{-1}\lambda^{-1})^{1-r} \leq \ol{\gamma}_{\psi} \lambda^{r-1}.
  \]
  For the second statement, for any $\lambda \in (0,\kappa^2]$, if $t \in [0,\lambda]$ we have by monotonicity of $\phi$:
  \[
    \abs{r_\lambda(t)\phi(t)} \leq \gamma_0 \phi(\lambda) \leq \ol{\gamma}_\psi \phi(\lambda),
  \]
  while for $t \in [\lambda,\kappa^2]$ we have by monotonicity of $\psi/\phi$ and
  $\psi$-qualification:
  \[
    \abs{r_\lambda(t)\phi'(t)} = \abs{r_\lambda(t)\psi(t)} \frac{\phi(t)}{\psi(t)} \leq
    \gamma_\psi \psi(\lambda) \frac{\phi(\lambda)}{\psi(\lambda)} \leq
    \ol{\gamma}_\psi \phi(\lambda);
  \]
  together these two cases bring the announced statement.
\end{proof}



Thus, if $\psi$ covers $\phi$ and $\psi$ is a qualification of $g_\lambda$, then so is
$\phi$.


\subsection{Effective Dimension and Empirical Effective Dimension}

The {\it effective dimension} is a key quantity for deriving learning rates in (semi-) supervised learning, parametrizing the effect  
of the input data through $\rho_X$, encapsulated in the second moment operator $B$.   
For $\lambda \in (0,\kappa^2]$ we set
\[ \NN(\lam) = \tr{\;( B+ \lam)^{-1}  B\;}  \; . \]
The empirical approximation, the {\it empirical effective dimension} 
\[ \NN_{\x}(\lam) = \tr{\;( B_{\x}+ \lam)^{-1}  B_{\x}\;}  \]
can be computed from a set of unlabeled input data $\x$. Just like $\NN(\lam)$ is crucial for finding an {\it a priori} parameter choice rule leading 
to optimal rates of convergence, the empirical $ \NN_{\x}(\lam)$ is essential for defining an {\it a posteriori} parameter choice rule.  
As functions of~$\lam>0$ both the effective dimension and the
empirical effective dimension are decreasing. Since $B$ is trace
class, we always have  $\NN(\lam) \leq \lam^{-1} \tr{B} \leq \lam^{-1} \kappa^2 $ as an
upper bound.

Moreover, for~$q>1$, and  since $t/(t+\lam/q) \leq qt/(t+\lam)$ for positive $t,\lam$,
it follows that
\begin{equation}
  \label{eq:nqlambda}
  \NN_\x(\lambda/q) \leq q \NN_\x(\lambda),\quad \lam>0,
\end{equation}
which will prove useful, below.

\section{Generalized Lepski\u\i{} principle}
\label{sec:basic-results}


Our goal is to establish an adaptive choice of the regularization
parameter~$\lam>0$ as this is achieved by the Lepski\u\i{}
principle. This works with a finite number of candidate estimates, and
we therefore fix a finite grid
\begin{equation}
  \label{eq:lamdam}
  \Lambda := \set{\lambda_{j},\quad  \kappa^2 \geq \lam_0 > \lam_1 > \ldots > \lam_m = \lam_{\min} > 0},
\end{equation}
with $m\geq 1$; we shall confine  its choice later. 


\pmathe{We fix the sample size~$n$.}
It will be transparent from the subsequent analysis that (a
priori) error estimates involve an (unknown, non-positive, non-decreasing)
function~$\lam\mapsto \A(\lam)$ \gillesb{(informally referred to as ``approximation error term'')}, and the (\pmathe{positive}, decreasing) function
\begin{equation}
   \label{eq:es-function}
   \es(n,\lam) := 
     \frac{\sigma\sqrt{\NN(\lam) \vee 1} + M/5}{\sqrt{ \lam n}},\quad \lambda
  >0,
\end{equation}
\gillesb{(informally referred to as ``estimation error term'')},
\pmathe{where the parameters~$M$ and~$\sigma$ are model specific.  Details for the roles
of~$M$ and~$\sigma$ (related to noise moments) will be given
in Section~\ref{sec:ass3}. These parameters are assumed to be known -- or an upper bound on it.}
  
Also, in order to access the prediction norm we shall use the isometry
 \begin{equation}
  \label{eq:isometry}
  \norm{g}_{L^{2}(\rho)} = \norm{B^{1/2}g}_{\Hr},\quad g\in\Hr.
\end{equation}
for which we also have an empirical counterpart for the empirical measure
$\wh{\rho}_n = \frac{1}{n} \sum_{i=1}^n \delta_{z_i}$:
\begin{equation}
 \label{eq:isometry-empirical}
 \paren{\frac{1}{n} \sum_{i=1}^n g^2(X_i)}^{\frac{1}{2}} =  \norm{g}_{L^{2}(\wh{\rho}_n)} = \norm{B_{\x}^{1/2}g}_{\Hr},\quad g\in\Hr.
\end{equation}
In particular, error bounds are given from 
\begin{equation}
  \label{eq:err-bounds-general}
 \norm{g}_{L^{2}(\rho)} = \norm{B^{1/2}g}_{\Hr}\leq \norm{\paren{B +
     \lam}^{1/2}g}_{\Hr},\quad g\in\Hr.   
\end{equation}
To apply the Lepski\u\i{} principle, the unavailable (population)
functions~$\es(n,\lam)$ and~$B$ will be replaced by their sample
version~$\es_{\x}(n,\lam)$ (i.e. wherein $\NN$ is replaced by $\NN_{\x}$) and~$B_{\x}$. With high probability, and
this will be formulated precisely later, quantities involving~$B_{\x}$
and~$\es_{\x}$ will be close to the population quantities.

In order to separate clearly the arguments, we will concentrate in this section on a purely deterministic argument underlying the Lepski\u\i{} principle, which will be first expressed in
a somewhat abstract version.
More precisely, we make the following assumption, \pmathe{where~$\Lam$ is the
grid specified in~(\ref{eq:lamdam})}.

\begin{assumption}
  \label{ass:error}
  There exists a positive self-adjoint operator $A$ and constant $C>0$ such that, for any $\lambda \in \Lambda$, there exists a
  function $\fl \in \Hr$ satisfying
  \begin{equation}
    \label{eq:assmp1}
    \norm{\paren{A + \lam}^{1/2}\paren{\frho - \fl}}_{\Hr} \leq C \sqrt \lam
  \paren{\A(\lam) +
  \es(\lam)},
\end{equation}
where the function~$\lam \in [0,\kappa^2] \mapsto \A(\lam) \in \R_+$ is non-decreasing with $\A(0) =
0$; and the function~$\lam\in [0,\kappa^2] \mapsto \sqrt{\lam} \es(\lam) \in \R_+$ is non-increasing.
\end{assumption}

The following estimate will be
further used in the sequel. It gives a consequence and interpretation
of~\eqref{eq:assmp1}. 
\begin{prop}
  \label{prop:interptool}
  Asssume $A$ is a positive self-adjoint operator on $\Hr$, and the element $h \in \Hr$ satisfies $\norm{(A+\lambda)^{\frac{1}{2}}h} \leq \sqrt{\lam} F(\lambda)$ for some $\lambda>0$ and
  \gillesb{some function $F:\R_+ \rightarrow \R_+$.} Then for any $s\in[0,\frac{1}{2}]$:
  \begin{equation}
    \norm{A^s h} \leq \lambda^s F(\lambda).
  \end{equation}
\end{prop}
\begin{proof}
  For $s=\frac{1}{2}$, it holds
  \[
    \norm[1]{A^{\frac{1}{2}}h} \leq
    \norm[1]{A^{\frac{1}{2}}(A+\lam)^{-\frac{1}{2}}}\norm[1]{(A+\lam)^{\frac{1}{2}}h} \leq \sqrt \lam F(\lambda).
  \]
  For $s=0$, we have
  \[
    \norm{h} \leq 
    \lam^{-\frac{1}{2}} \norm[1]{ \lam^{\frac{1}{2}}h}
\leq \lam^{-\frac{1}{2}}
\norm[1]{(A+\lam)^{\frac{1}{2}}h}
\leq F(\lam).
\]
Finally, by interpolation between the two last inequalities, for $s\in [0,\frac{1}{2}]$:
\[
\norm{{A}^{s}h}
\leq \norm{h}^{1-2s} \norm[1]{{A}^{\frac{1}{2}}h}^{2s}
\leq  \lam^{s} F(\lam).
\]
\end{proof}

\medskip

Now we can formulate the parameter choice in this abstracted framework; it is based on a modified Lepski\u\i{}'s method (or "balancing principle").
\begin{defi}[Parameter choice, abstract version]
\label{def:lepest_abs}
We consider the notation from Assumption~\ref{ass:error}. For the grid~$\Lambda$
from~(\ref{eq:lamdam}) we set
\begin{align}
  \label{eq:setlepski_abs}
    \M(\Lambda) := \bigg\{\lam  \in \Lambda \; :\;  %
   \norm{ (A + \lam')^{1/2}(\fl - f^{\lam'})   }_{\Hr}
         \leq   4 C \sqrt{\lam'} \es(\lam'), \; 
          \forall \lam' \in \Lambda, \text{ s.t. } \lam' \leq \lam\;    \bigg\}.
\end{align} 
The balancing parameter is given as
\begin{equation}
\label{hatlam} 
\hat \lam :=\max \; \M(\Lam)  \;;
\end{equation}
obviously this quantity is always well-defined since $\lam_{\min} \in \M(\Lam)$.
\end{defi}

  This parameter choice has the remarkable property of simultaneous adaptivity
  in several norms $\norm{A^{s}\cdot}_{\Hr}$ for~$0 \leq s \leq 1/2$.
  This is formulated in the following main result.

Let~$\lamstar$ be given as
\begin{equation}
  \label{eq:lamstar-new2}
  \lamstar := \max \paren{\set{\lam\in \Lam,\quad \A(\lam) \leq \es(\lam) } \cup \set{\lam_{\min}}}.
\end{equation}
The value~$\lamstar$ is unknown to us, since the function~$\A$
is. 

\begin{theo}
  \label{theo:abstractoracle}
Suppose Assumption~\ref{ass:error} holds.
For the parameters~$\hat\lam$ and~$\lamstar$
from~\eqref{hatlam} and~\eqref{eq:lamstar-new2}, respectively, the
following holds true:
\begin{enumerate}
\item \label{it:1}
  The parameter~$\hat\lam$ satisfies~$\hat\lam \geq
  \lamstar$, and \pmathe{it obeys}
  \begin{equation}
    \label{eq:hatlam-bound3}
    \norm[1]{(A+\hat\lam)^{\frac{1}{2}}(\fd - f^{\hat\lam})} \leq 6 C \sqrt{\lamstar} \wt{\es}(\lamstar),
  \end{equation}
    where $\wt{\es}(\lamstar) := \max(\es(\lamstar),\A(\lamstar))$.
\item  \label{it:2}
  If it holds that $\es(\lam_k) \leq C_{\es}\es(\lam_{k-1})$ for $k=1,\ldots,m$ for
  some constant~$C_{\es}>1$, then for any $s\in[0,\frac{1}{2}]$:
  \begin{equation}
    \label{eq:oracle-bound2}
    \lamstar^s  \wt{\es}(\lamstar)
    \leq 
    C_{\es} \min_{\lam\in [\lam_{\min},\lam_0]}\set[2]{
       \lam^s \paren{\A(\lam) + \es(\lam)}}.
  \end{equation}
\end{enumerate}
(Note that $\wt{\es}(\lamstar) \neq \es(\lamstar)$ can only possibly happen in the
``edge'' case $\lamstar = \lambda_{\min}$.)
\end{theo}
Observe that, via Proposition~\ref{prop:interptool}, the estimate~\eqref{eq:hatlam-bound3} leads
to a control of $\norm[1]{\fd - f^{\hat\lam}}_\Hr$ as well as (in the case $A$ could be formally taken equal to the operator $B$ defined in Section~\ref{sec:defops}) $\norm[1]{B^{\frac{1}{2}}(\fd - f^{\hat\lam})}_\Hr=\norm[1]{\fd - f^{\hat\lam}}_{L^2(\rho)}$, which thanks to the second estimate~\eqref{eq:oracle-bound2}, can be interpreted as an oracle-type inequality for both of
these error measures.
\begin{proof}[Proof of Theorem~\ref{theo:abstractoracle}]
  Assume first that $\lamstar>\lam_{\min}$, so that $\A(\lamstar) \leq \es(\lamstar)$ holds. Then for any~$\lam \in \Lambda$ with~$\lam  \leq \lamstar$ we find that
  \begin{multline*}
    \norm[1]{(B+\lam)^{\frac{1}{2}}( \fl - \fls)}\\
    \begin{aligned}
    & \leq \norm[1]{(B+\lam)^{\frac{1}{2}}(\fl - \fd)} + \norm[1]{(B+\lam)^{\frac{1}{2}}(\fls - \fd)} \\
    & \leq \norm[1]{(B+\lam)^{\frac{1}{2}}(\fl - \fd)} +
    \norm[1]{(B+\lam)^{\frac{1}{2}}(B+\lamstar)^{-\frac{1}{2}}}\norm[1]{(B+\lamstar)^{\frac{1}{2}}(\fls - \fd)} \\
& \leq C \sqrt{\lam}\paren{\A(\lam) + \es(\lam)} +
   C \sqrt{\lamstar} \paren{\A(\lamstar) + \es(\lamstar)}\\
&\leq 2 C\sqrt{\lam}\es(\lam) + 2 C \sqrt{\lamstar}\es(\lamstar)\\
& \leq 4 C\sqrt{\lam}\es(\lam),
\end{aligned}
\end{multline*}
so that $\lamstar$ belongs to the set appearing in \eqref{eq:setlepski_abs}, and thus~$\hat\lam \geq \lamstar$. Finally, in the initially excluded case $\lamstar=\lam_{\min}$, obviously this conclusion still holds.
This in turn yields
\begin{align*}
  \norm[1]{(B+\lamstar)^{\frac{1}{2}}(\fd - f^{\hat\lam})}
  & \leq   \norm[1]{(B+\lamstar)^{\frac{1}{2}}(\fd - f^{\lamstar})} +
    \norm[1]{(B+\lamstar)^{\frac{1}{2}}(f^{\lamstar}- f^{\hat\lam})}
  \\
  & \leq 2 C \sqrt{\lamstar} \wt{\es}(\lamstar) + 4  C\sqrt{\lamstar} \es(\lamstar)\\
  & \leq
  6  C \sqrt{\lamstar} \wt{\es}(\lamstar),
\end{align*}
using \eqref{eq:assmp1} and \eqref{eq:setlepski_abs} for the first and second term,
respectively, establishing~\eqref{eq:hatlam-bound3}.

For the proof of the oracle property~\eqref{eq:oracle-bound2},
define $\lamstar^+:= \min\set{\lam' \in \Lambda: \lam' > \lamstar}$
(which is well-defined, provided $\lamstar < \lam_0 = \max \Lambda$;
exclude for now the case $\lamstar = \lam_0$).
Observe that $\lamstar^+ > \lamstar$ must correspond to two consecutive indices in
$\Lambda$, so that by assumption we have 
$\es(\lamstar) \leq C_\es \es(\lamstar^+)$; also by definition of $\lamstar$, we
have $\A(\lamstar^+) > \es(\lamstar^+)$.

We consider two cases for $\lam \in [\lam_{\min},\lam_0]$:
\begin{itemize}
\item $\lam \leq \lamstar^+:$
In this case, noticing that
$\lambda \mapsto \lambda^s\es(\lam)$ is non-increasing for $s\in[0,\frac{1}{2}]$,
\begin{align*}
     \lam^s \paren{\A(\lam) + \es(\lam)} & \geq  \lam^s \es(\lam) \geq  (\lamstar^+)^s \es(\lamstar^+) \geq C_\es^{-1} \lamstar^s \es(\lamstar).
\end{align*}
\item $\lam > \lamstar^+:$
  We can then bound
  \begin{align*}
    \lam^s \paren{\A(\lam) + \es(\lam)}  \geq  \lam^s \A(\lam)
   \geq    \lamstar^s \A(\lamstar^+) 
   &\geq   \lamstar^s \es(\lamstar^+)
   \geq C_{\es}^{-1}   \lamstar^s \es(\lamstar).
  \end{align*}
\end{itemize}
This establishes~\eqref{eq:oracle-bound2} if $\wt{\es}(\lamstar) = \es(\lamstar)$;
  otherwise, it must be the case that $\wt{\es}(\lamstar) = \A(\lamstar)$ and $\lamstar = \lam_{\min}$, in which case~\eqref{eq:oracle-bound2} obviously holds since $\A$ is non-decreasing.

Finally coming back to the edge situation $\lamstar = \lam_0 = \max \Lambda$ first put aside above,
we then only have to consider $\lambda \leq \lamstar$ and a straightforward
modification of the argument in the first case above yields
$\lam^s \paren{\A(\lam) + \es(\lam)} \geq \lamstar^s \es(\lamstar)$.
\end{proof}


\section{Bounds in the supervised learning setting}
\label{sec:ass3}

\pmathe{In order to validate the usage of the generalized Lepski\u\i{}
principle from Section~\ref{sec:basic-results} we now need to make an
assumption of the distribution of the noise~$\varepsilon$ in the
model~(\ref{eq:base}),}
and we make the following Bernstein-type moment inequality assumption:
\begin{assumption}
  \label{ass:bernstein}
  The data generating distribution is such that there exists positive constants $\sigma, M$ with
  \begin{equation}
    \E_{(X,Y)\sim\rho}{(Y-f_\rho(X))^k} \leq \frac {k!}{2} \sigma^2 M^{k-2},
  \end{equation}
  for all integers $k\geq 2$.
\end{assumption}
Note again that we constantly assume that the constants $\sigma,M$ are
{\em known}, \pmathe{or at least valid upper bounds}. \gillesb{An upper confidence bound holding with
large probability would also be suitable, but we don't touch the subject of estimating
noise variance in this paper.}

In the abstract version of the (modified) Lepski\u\i's method presented in the previous section,
taking formally $A:=B$ would lead to a parameter choice depending on unobserved population quantities, which would not be an {\em a posteriori} choice.
In the present section, we turn to establishing an error bound for an estimator which only
depends on observable quantities (and of quantities which are assumed to be known such as $M$ and
$\sigma$). For this, the gist of the approach is to establish that 
Assumption~\ref{ass:error} holds on an event of high probability,
for entirely empirical quantities;
then, provided this event is satisfied, we will be able to apply
Theorem~\ref{theo:abstractoracle}.

Let us specify some definitions first. For the sake of simplicity, we will consider a geometrically regular grid of factor $q>1$, i.e. \gillesb{the set of candidate regularization parameters will be
  \[
    \Lambda^{(q)} := \set{\lambda_i = \kappa^2 q^{-i}, i \in \mathbb{N} }.
    \]
  Note that only finite subsets of the above grid, with a minimum element either deterministic or
  data-dependent, will be considered in the sequel.}
In the rest of this section we assume $q$ to be fixed.
Since our estimates will be based on exponential deviation probabilities, we will denote
\[
  \letan := 2 \log \paren{ \frac{8 \log n}{\eta \log q}},
\]
where $\eta\in(0,1)$ will denote the (small) probability of the favorable event not being satisfied, and which is assumed to be fixed a priori. For practical purposes, one can for instance think of
$\eta$ as being some negative power of the number of training examples $n$, so that $\letan$ is
a logarithmic factor in $n$; alternatively, if a fixed small probability of failure $\eta$ is deemed
acceptable for any $n$, the factor $\letan$ is only $\mathcal{O}(\log \log n)$.

Let
\begin{equation}
   \label{eq:esx-function}
   \es_\x(n,\lam) := 
     \frac{\sigma\sqrt{2(\NN_\x(\lam) \vee 1)} + M/5}{\sqrt{\lam n}},\quad \lambda
  >0,
\end{equation}
be an empirical version of~$\es(\lambda)$ from~(\ref{eq:es-function}) (the numerical constants are for technical
convenience but don't have any special meaning here). We will also define a data-dependent (and therefore random) grid of parameters. The motivation for this is that the main error estimate involves
$\sqrt{\lambda}\es_\x(\lam)$ as a term, so one might as well restrict the search to the region where this term is at most of order 1. We thus define
\pmathe{\begin{equation}
  \label{eq:gridz}
  \Lambda_{\x} := \set{ \lambda\in\Lam^{(q)} \ \text{ s.t. }\;
    \lambda \geq 100\kappa^2 \letan^2/n \text{ and } \lambda \geq  3 \kappa^2
    (\NN_\x(\lambda)\vee 1)/n  }.
\end{equation}
}
In principle, it could happen that $\Lambda_\x$ is empty, in which
case the procedure below will be undefined (formally, the parameter~$\hat \lam_\z$ can be taken equal to $\kappa^2$ in that case). This does
not contradict the performance bounds to come, since it will be clear from their proof
that at least for $n$ big enough then with high
probability the grid~$\Lam_\x$ will be  non empty.

We can now define the purely data-driven parameter choice.
\begin{defi}[Data-driven parameter choice]
  Let $(\fzl)_{\lambda \in (0,\kappa^2]}$ be a family of regularized estimates
    as defined in~\eqref{eq:fzl} using a regularization function of qualification $\psi$,
    and $\ol{\gamma}_\psi$ be defined as in Proposition~\ref{prop:interpreg}.
\label{def:lepest}
For the grid~$\Lambda_\x$ from~\eqref{eq:gridz}, we set
\begin{multline}
  \label{eq:setlepski}
    \M_\z(\Lambda_\x) := \bigg\{\lam  \in \Lambda_\x \; :\;  %
   \norm{ (B_{\x} + \lam')^{1/2}(\fzl - f^{\lam'}_\z)   }_{\Hr}
         \leq   64 \ol{\gamma}_\psi \letan \sqrt{\lam'} \es_\x(\lam'), \;\\ 
          \forall \lam' \in \Lambda_\x, \text{ s.t. } \lam' \leq \lam\;    \bigg\}.
\end{multline} 
The data-driven balancing parameter is given as
\begin{equation}
\label{hatlamz} 
\hat \lam_\z :=\max \M_\z(\Lambda_\x). 
\end{equation}     
\end{defi}

\begin{rem}
  The parameter $\hat \lam_\z$ is indeed an {\em a posteriori} choice since it only depends on
  quantities that are assumed known to the user and empirical quantities. Furthermore,
  the computation itself is relatively easy since the norm appearing in~\eqref{eq:setlepski}
  can be rewritten as
  \[
    \norm{ (B_{\x} + \lam')^{1/2}(f_{\z}^{\lam} - f_{\z}^{\lam'})}_{\Hr}
    = \paren[3]{  
      \frac{1}{n} \sum_{i=1}^n (f_{\z}^{\lam} - f_{\z}^{\lam'})^2(X_i)
       + \lambda' \norm{f_{\z}^{\lam} - f_{\z}^{\lam'}}^2_{\Hr} }^{\frac{1}{2}}.
   \]
   In practice these quantities can directly be computed, since the estimators $f_\z^\lam$ are
   represented as a kernel expansion $\fzl = \sum_{i=1}^n
   \alpha^\lam_i K_{x_i}$ (see~\eqref{eq:shift}), so that 
   $\fzl(x) = \sum_{i=1}^n \alpha^\lam_i K(x_i,x)$, and then
   \[ \norm{f_{\z}^{\lam} - f_{\z}^{\lam'}}^2_\Hr = \sum_{i,j =1}^n (\al_i^\lam-\al^{\lam'}_i)
     (\al_j^\lam-\al^{\lam'}_j) K(x_i,x_j).\]
\end{rem}

To analyze the behavior of the proposed algorithm, we assume a general
form source condition for the target function,
\begin{equation}
  \label{eq:source}
  \fo = \phi(B) h, \qquad \norm{h}_\Hr \leq 1,
\end{equation}
where $\phi: [0,\kappa^2] \rightarrow \R_+$ \pmathe{is an index
  function as introduced in Definition~\ref{def:index}, which can be decomposed as}~$\phi(t) = \varphi_1(t)
\varphi_2(t)$, and $\varphi_1$, $\varphi_2$ are nondecreasing functions 
$[0,\kappa^2] \rightarrow \R_+$ such that $\varphi_1$ is $\lipc$-Lipschitz and
$\varphi_2$ is sublinear, by which we mean that $\varphi_2(t)/t$ is nonincreasing
for $t>0$.
We observe that in that setting it would be redundant to
consider the more general inequality~$\norm{h}_\Hr \leq R$,
since this can always be achieved by implicit rescaling~$\phi \to R\phi$.

In order to guarantee a control of the \pmathe{function~$\A(\lam)$ from Assumption~\ref{ass:error}, playing he role of approximation term} in our main probabilistic estimate), we will finally make the assumption that the regularization method has qualification $t \mapsto \sqrt{t}\phi(t)$
(see Definition~\ref{def:qualif}). We recall that qualification with a covering function is sufficient, so that this assumption has to be interpreted as that of a {\em minimal qualification}. Equivalently, if the qualification of the method is fixed and equal to $\psi$ (assumed to
cover $t\mapsto \sqrt{t}$),
the results to come hold for all source conditions~\eqref{eq:source} covered by $t\mapsto \psi(t) /\sqrt{t}$.

The next theorem is our main result:
\begin{theo}
  \label{theo:mainlearning}
  Assume the observed data $\z=(\x, \y)=((x_1, y_1),...,(x_n,y_n))$ is drawn independently and identically distributed according to $\rho$, satisfying Assumption~\ref{ass:bernstein},
  wherein parameters $\sigma$ and $M$ are assumed to be known.
  
  Let $\fzl$, $\lam>0$ be defined as in Section~\ref{sec:regularization} for a regularization
  family of qualification $\psi$, where $\psi$ is assumed to cover
  $t\mapsto \sqrt{t}$. Let $\ol{\gamma}_\psi$ be defined as in
  Proposition~\ref{prop:interpreg}. Let $q>1$, $\eta\in(0,1)$ be fixed
  and $\hat\lam_\z$ be the
  data-dependent, a posteriori parameter choice given by~\eqref{hatlamz}.


  If the regression function $f_\rho$ satisfies a source condition of the form given by~\eqref{eq:source} with index function $\phi$ covered by $t\mapsto \psi(t)/\sqrt{t}$,
  then with probability at least $1-\eta$ it holds for all $s\in[0,\frac{1}{2}]$:
  \begin{align}
    \label{eq:trueoracle}
    \norm[1]{B^s(\fd - f^{\hat\lam_\z})} \leq    c q \ol{\gamma}_\psi \letan^3 \min_{\lam\in[ \lam_{\min}, \kappa^2]}\paren{
    \lam^s \paren[2]{ \phi(\lambda) + \es(n,\lam) + \frac{\lipc \varphi_2(\kappa^2)\kappa^2}{\sqrt{n}} }},
  \end{align}
  where  $c$ is a numerical constant ($c=384$ works), and
  \begin{equation}
    \label{eq:lambdamin}
    \lambda_{\min} = q \min\set{ \lambda: \lambda \geq 100\kappa^2 \letan^2/n \text{ and }
      \lambda \geq  6 \kappa^2 (\NN(\lambda)\vee 1)/n},
    \end{equation}
    where we assume $n$ large enough so that $n\geq \max(100
    \letan^2,6)$, ensuring the above minimum to be well-defined.
 \end{theo}
Observe that this result takes the form of an oracle inequality for the norms
$\norm{B^s.}$ with respect to the population operator $B$ (we recall that this includes the
prediction norm $\norm{.}_{L^2(\rho)}$ for $s=\frac{1}{2}$, see~\eqref{eq:isometry});
and with respect to
a {\em deterministic} parameter range where $\lambda_{\min}$ is determined from the
true effective dimension $\NN$. The next corollary shows that the covered
range is large enough to include a parameter $\lambda_n$ leading to
optimal convergence rates in most situations 
(again, provided that sufficient qualification holds).
\begin{cor}
  \label{cor:rate}
  Assume the same conditions as in Theorem~\ref{theo:mainlearning}.
  Assume $\lim_{\lambda \rightarrow 0} \NN(\lambda) = \infty$ and $\phi(\lambda)>0$ for $\lambda>0$ (true infinite-dimensional setting); assume also that $\phi^2(\lambda)/\NN(\lambda) = o_{\lambda}((\log \log \lambda^{-1})^{-2})$ holds (as $\lambda \rightarrow 0$).
  
  Define $\comfun(\lam) := \lam \phi^2(\lam)/\NN(\lam)$, which is a continuous (strictly) increasing function on $[0,\kappa^2]$ satisfying $\comfun(0)=0$.

Fix $\eta\in(0,1)$.  Then for $n$ larger than a certain $n_0$ (depending on all parameters), with probability at least $1-\eta$ it holds that for any $s\in[0,\frac{1}{2}]$:
\begin{equation}
  \label{eq:therate}
  \norm[1]{B^s(\fd - f^{\hat\lam_\z})} \leq C_\triangle (\log \eta^{-1} + \log \log n)^3 \lam_n^s \phi(\lambda_n), \text{ where }  \lam_n:= \comfun^{-1}\paren{\frac{\sigma^2}{n}},    
\end{equation}
where $C_\triangle$ is a factor depending on all parameters but $\sigma$.
\end{cor}

As a standard example, if the spectrum of $B$ satisfies a power
decay with exponent $b < 1$ (which entails $\NN(\lam) = \Theta(\lam^{-b})$ as $\lam \rightarrow 0$), and the target function satisfies a H\"older source condition, i.e. $\phi(t) = c t^r$, then the upper bound in~\eqref{eq:therate} is
of order $\mathcal{O}(n^{-\frac{r+s}{2r + 1 + b}} (\log \log n)^3)$, which
is known to be minimax optimal up to the double logarithmic factor, and
this rate is obtained adaptively without knowing a priori the values of $r$ nor $b$ (but assuming sufficient qualification of the regularization method). 
\gillesb{Without any assumption
  on the spectrum decay, it always holds $\NN(\lam) \leq \lam^{-1} \tr{B} \leq \kappa^2/\lambda$,
  so $\Delta(\lam) = \Omega(\lam^2 \phi^2(\lam))$ and $\Delta^{-1}(u) = \mathcal{O}(\xi^{-1}(u))$,
  where $\xi(\lam) := \lam^2 \phi^2(\lam)$. Under the same H\"older source condition as above,
  we are therefore always guaranteed a convergence rate at least $\mathcal{O}(n^{-\frac{r+s}{2(r + 1)}} (\log \log n)^3)$ under sufficient qualification.}

\begin{proof}[Proof of Corollary~\ref{cor:rate}]
  Define $\lam_{n} := \comfun^{-1}(\sigma^2/n)$, which is well-defined as soon
  as $n$ is big enough to ensure $\sigma^2/n \leq \comfun(\kappa^2)$. Since $\comfun^{-1}$ is a (strictly) increasing continuous function equal to 0 in 0, it holds that $\lam_n = o_n(1)$.
  By the assumption $\lim_{\lambda \rightarrow 0} \NN(\lambda)=\infty$, we conclude $N(\lam_n)\vee 1=
  \NN(\lam_n)$ for $n$ large enough, which we assume from now on.
  Since $\comfun(\lambda_n) = \sigma^2/n$, we have $\NN(\lam_n)/(n\lam_n) = \sigma^{-2} \phi(\lam_n) = o_n(1)$,
  so that $\lam_n$ 
  satisfies the second condition in the set defining $\lam_{\min}$ in equation~\eqref{eq:lambdamin}
  for $n$ big enough, which we assume also to be the case from now on.

  To check that $\lam_n$ 
  satisfies the first condition in the set appearing in~\eqref{eq:lambdamin},
  observe that for $n$ big enough, putting $\eps_n:= 100 q \kappa^2 \letan^2/n = \Theta((\log \log n)^2/n)$, we have for any $C>0$ and $n$ large enough, from the assumption $\phi^2(\lambda)/\NN(\lambda) = o_{\lambda}( (\log \log \lambda^{-1})^{-2})$:
  \[
    \comfun(\eps_n) = \frac{\phi^2(\eps_n)}{\NN(\eps_n)} \eps_n
    < C (\log \log \eps_n^{-1})^{-2} \eps_n \leq \frac{ \sigma^2}{n},
  \]
  where the last inequality holds if we choose $C$ appropriately small enough. From the
  above inequality we deduce $\lam_n = \comfun^{-1}(\frac{ \sigma^2}{n}) > 100 q \kappa^2 \letan^2/n$; thus we have ensured 
  $\lam_n > \lambda_{\min}$ for $n$ large enough. With this choice of parameter in the right-hand side of \eqref{eq:trueoracle}, using $\lambda_n \rightarrow 0, \NN(\lambda_n) \rightarrow \infty$
  we deduce 
  $S(n,\lam_n) \leq C_\triangle \sqrt{\sigma^2 \NN(\lambda_n)}/{\sqrt{\lambda_n} n} = C'_\triangle \phi(\lambda_n)$, ${\lipc \varphi_2(\kappa^2)}\kappa^2/{\sqrt{n}} = o_n(S(n,\lambda_n))$, 
  leading to the announced conclusion.
\end{proof}

We now give the main steps of the proof to establish
Theorem~\ref{theo:mainlearning}, with technical results relegated to the Appendix.
As announced previously, the cornerstone of the analysis is to relate empirical and
population quantities. The most useful probabilistic estimates are summarized in the following proposition.
In order to give all estimates a similar form, we introduce the notation
\begin{equation}
  \label{eq:bnab}
  \bnl(a,b) := \frac{1}{\sqrt{\lambda}}
   \paren{a\frac
 {2\kappa}{\sqrt{\lambda}n} + b\sqrt{\frac{\NN(\lam)}{n}}}.
\end{equation}
\begin{prop}
  \label{prop:probestimates_main2}
  Suppose Assumption~\ref{ass:bernstein} holds.
  Let $\lambda>0$ be fixed, and $\eta \in (0,1)$, and put $\leta:=2 \log ( 8 / \eta)$.
  There exists an event $\Omega_{\lambda,\eta}$ of probability at least $1 - \eta$,
  such that the following estimates hold simultaneously:
%
  \begin{itemize}
    \item Multiplicative operator perturbation bound:
  \begin{equation}
  \label{eq:phipertl2}
  \norm{\paren{ \lam + B}^{\frac{1}{2}}(\lam + B_{\x})^{-\frac{1}{2}}} 
  \leq
\leta \paren{\bnl(\kappa,\kappa) +1}.
\end{equation}
\item Multiplicative effective dimension bound:
if $\lambda$ is such that $\lambda\geq \frac{4\kappa^2}{n}$ holds:
   \begin{equation}
       \label{eq:reldeveffdim2}
       \max\paren{\frac{\NN(\lam)\vee 1}{\NN_{\x}(\lam)\vee 1}, \frac{\NN_{\x}(\lam)\vee 1}{\NN(\lam)\vee 1}}
       \leq \paren{1 +\frac{4\kappa \leta}{\sqrt{\lam n}}}^2.
   \end{equation} 
%
%
   \item Main estimate bound:
  if the target function $f_\rho$ satisfies a source condition with
  index function $\phi$ as detailed in \eqref{eq:source} and such that the
  qualification $\psi$ of the regularization covers $t\mapsto \sqrt{t}\phi(t)$,
  and if $\lambda$ is such that $\lambda \geq 100 \kappa^2 \leta^2/n$, then:
  \begin{equation}
    \label{eq:assmpt1prob}
 \norm{\paren{B_{\x} + \lam}^{1/2}\paren{\fo - \fzl}}_{\Hr}  \leq 
 4 \ol{\gamma}_\psi \paren{1 + \bnl(\kappa,\kappa)}^2 \leta^2 \sqrt \lam \paren{\lipc \varphi_2(\kappa^2) \kappa^2 n^{-\frac{1}{2}}  + \phi(\lam)
   + \bnlab {M}{\sigma}}.
\end{equation}
\end{itemize}
\end{prop}
Observe that estimate~\eqref{eq:assmpt1prob} takes the form of Assumption~\ref{ass:error} for empirical quantities, holding with high probability; while
estimates~\eqref{eq:phipertl2} and~\eqref{eq:reldeveffdim2} will allow us to go from
an oracle inequality with respect to empirical quantities back to population quantities.
The proof for Proposition~\ref{prop:probestimates_main2} 
is relegated to Appendix~\ref{sec:proof-prop2}; with these results at hand
we turn to proving the main result. 

\begin{proof}[Proof of Theorem~\ref{theo:mainlearning}]
  Define the (deterministic) grid
 \pmathe{ \[
    \Lambda_0 := \set{ \lambda \in\Lam^{(q)} \ \text{ s.t. }\;
      \lambda\geq 100\kappa^2 \letan^2/n}.
  \]
  }
  The first step is to ensure the validity of estimates appearing in Proposition~\ref{prop:probestimates} 
  simultaneously for all
  $\lambda \in \Lambda_0$. This is achieved by a simple union bound, and since the cardinality
  of $\Lambda_0$ is bounded as $\abs{\Lambda_0} \leq 1 + \frac{\log n - \log (100\letan^2)}{\log q} \leq (\log n)/(\log q)$ (since $\letan\geq 1$), we obtain estimates uniformly valid
  over $\lambda \in \Lambda_0$ if, in Proposition~\ref{prop:probestimates}, 
  we replace $\eta$ by $\eta (\log n) / (\log q)$, and in consequence
  $\leta$ by $\letan$. For the rest of this proof, we assume to be on the corresponding
  high probability event where all estimates hold.
  
  For all $\lambda \in \Lambda_0$ estimate~\eqref{eq:reldeveffdim2} yields
  \begin{equation}
    \label{eq:nnbound}
    \max\paren{\frac{\NN(\lam)\vee 1}{\NN_{\x}(\lam)\vee 1}, \frac{\NN_{\x}(\lam)\vee 1}{\NN(\lam)\vee 1}}
    \leq \paren[3]{1 +\frac{4\kappa \letan}{\sqrt{\lam n}}}^2 \leq 2.
  \end{equation}
  This estimate holds in particular for all $\lambda \in \Lambda_\x \subset \Lambda_0$, and
  the second condition in the definition~\eqref{eq:gridz} of $\Lambda_\x$ combined with the above yields
  \begin{equation}
    \forall \lambda \in \Lambda_\x: \qquad
    \bnl(\kappa,\kappa) = 
   \paren{\frac
     {2\kappa^2}{\lam n} + \sqrt{ \frac{\kappa^2\NN(\lam)}{\lam n}}} \leq \paren{\frac{1}{50}
   + \sqrt{\frac{\NN(\lam)}{3(\NN_\z(\lam)\vee 1)}}} \leq 1.
  \end{equation}
  In turn, this yields from estimate~\eqref{eq:assmpt1prob} that for all $\lambda \in \Lambda_\x$:
  \begin{align*}
    \norm{\paren{B_{\x} + \lam}^{1/2}\paren{\fo - \fzl}}_{\Hr}
    & \leq 
16 \ol{\gamma}_\psi  \letan^2 \sqrt \lam \paren{ \frac{\lipc \varphi_2(\kappa^2)\kappa^2}{\sqrt n}  +  \phi(\lam)
                                                               + \frac{2\kappa M}{\lambda n} + \sqrt{\frac{\sigma^2\NN(\lambda)}{\lam n}}}\\
    &\leq 
16 \ol{\gamma}_\psi  \letan^2 \sqrt \lam \paren{ \frac{\lipc \varphi_2(\kappa^2)\kappa^2}{\sqrt n}  + \phi(\lam)
 + \frac{M}{5\sqrt{\lambda n}} + \sqrt{\frac{2\sigma^2(\NN_\z(\lambda) \vee 1)}{\lam n}}}.
\end{align*}
Thus, Assumption~\ref{ass:error} is satisfied, with $A:= B_\x$,
$\A(\lam) := \phi(\lam) + \lipc\varphi_2(\kappa^2)\kappa^2/\sqrt{n}$, $\es(\lam) := \es_\x(n,\lambda)$,
and $C:=16 \ol{\gamma}_\psi \letan^2$.

The choice~\eqref{hatlamz} of $\hat\lam_\z$ corresponds exactly to~\eqref{hatlam} with the quantities defined above.
 Thus the assumptions of Theorem~\ref{theo:abstractoracle} are  satisfied for the quantities
defined above, and we deduce from point (1) of the theorem:
\[
  \norm[1]{(B_\x+\lamstar)^{\frac{1}{2}}(\fd - f^{\hat\lam})} \leq 6 C \sqrt{\lamstar} \wt{\es}_\x(n,\lamstar),
\]
where $\lambda^*$ is as defined in~\eqref{eq:lamstar-new2},
and $\wt{\es}_{\x}(n,\lamstar) = \max(\es_{\x}(n,\lamstar),\A(\lamstar))$. This implies, using~\eqref{eq:phipertl2}:
\[
  \norm[1]{(B+\lamstar)^{\frac{1}{2}}(\fd - f^{\hat\lam})}
  \leq   \norm[1]{(B+\lamstar)^{\frac{1}{2}}(B_\x+\lamstar)^{-\frac{1}{2}}} \norm[1]{(B_\x+\lamstar)^{\frac{1}{2}}(\fd - f^{\hat\lam})}
  \leq 12 \letan C \sqrt{\lamstar} \wt{\es}_\x(n,\lamstar);
\]
then, using Proposition~\ref{prop:interptool}, for any for any $s\in[0,\frac{1}{2}]$ we have:
\[
  \norm{B^s(\fd - f^{\hat\lam})} \leq 12 \letan C \lamstar^s \wt{\es}_\x(n,\lamstar).
\]
\pmathe{By using the estimate~(\ref{eq:nqlambda}) we find for two successive elements
of $\Lambda_\x$ that $\NN_\x(\lam_{i}) \leq q \NN_\x(\lambda_{i-1})$,
and consequently that~$\es_\x(n,\lambda_i) \leq q
\es_\x(n,\lambda_{i-1})$. This allows to apply item~(\ref{it:2}) of
Theorem~\ref{theo:abstractoracle}.}

Thus, for any for any $s\in[0,\frac{1}{2}]$:
\begin{align*}
  \norm{B^s(\fd - f^{\hat\lam})}
  & \leq 12 \letan C q \min_{\lambda \in [\min \Lam_\x, \kappa^2]} \set{\lambda^s(\A(\lambda) + \es_\x(n,\lambda)}.
\end{align*}
To finish the proof, it suffices to note that $\min \Lambda_\x \leq \lam_{\min}$
(as defined by~\eqref{eq:lambdamin}) as well as
$\es_\x(n,\lambda) \leq 2\es(n,\lambda)$, both as a consequence of~\eqref{eq:nnbound}.
\end{proof}

\appendix



\section{Probabilistic estimates for fixed regularization parameter}     
\label{sec:prob-estimates}

Here we derive some standard probabilistic bounds in a novel
form. Although the main ingredient is the standard probabilistic bound
for Hilbert-space random variables, the novel approach highlights the
structure of the estimates more clearly.

For the analysis in this section, we assume $\lambda>0$ has been fixed; it is
then convenient to introduce the ``standardized'' quantities
$S:= \lam^{-1}B$, $S_{\x} := \lambda^{-1}B_{\x}$, and
$\bar K_{x} := \lambda^{-1/2}K_{x}$.
Finally, let $\varphi:\R_+ \rightarrow \R_+$ be a
nondecreasing and sublinear function.
We introduce shorthand notation for some key quantities (the default norm for
operators is the operator norm, while the index HS indicates
Hilbert-Schmidt norm), where a sample~$\z = (\x,y)$ is fixed. We let
\begin{align}
  \Gamma_{\x} & := \norm{S-S_{\x}}_{\hs}; \label{it:gammax}\\
  \Psi_{\x} &:= \norm{\paren{I + S}^{-1/2}(S - S_{\x})}_{\hs}; 
\label{it:psix} \\
  \Xi^\varphi_{\x} & := \norm{\varphi\paren{I + S}\varphi\paren{I + S_{\x}}^{-1}} ; \\
  \Theta_{\z}& := \norm{ \paren{I +
  S}^{-1/2} \frac 1 n \sum_{j=1}^{n} (y_{j} - \frho(x_{j}))\bar K_{x_{j}}};\\
  \NN = \NN(\lambda) & := \tr{\paren{I + S}^{-1}S};\label{it:effdim}\\
  \NN_\x = \NN_\x(\lambda) & := \tr{\paren{I + S_\x}^{-1}S_\x}.\label{it:effdimemp}
\end{align}

We repeat here for convenience the notation introduced in~\eqref{eq:bnab}:
\begin{equation}
  \label{eq:bnab2}
  \bnl(a,b) := \frac{1}{\sqrt{\lambda}}
   \paren{a\frac
 {2\kappa}{\sqrt{\lambda}n} + b\sqrt{\frac{\NN}{n}}},
\end{equation}
allowing to give all bounds a similar form.
The following proposition states the estimates needed for our analysis.
It subsumes the two first estimates of Proposition~\ref{prop:probestimates_main2}
in the main text.
\begin{prop}
  \label{prop:probestimates}
  Suppose Assumption~\ref{ass:bernstein} is satisfied.
  Let $\lambda>0$ be fixed, and $\eta \in (0,1)$, and put $L:=2 \log ( 8 / \eta)$.
  There exists an event $\Omega_{\lambda,\eta}$ of probability at least $1 - \eta$,
  such that the following estimates hold simultaneously:
%
  \begin{align}
        \label{eq:devHShoeff}
  \Gamma_{\x,\lam} := \lam\Gamma_{\x} = \norm{B - B_{\x}}_{\hs}
    & \leq  L \frac{\kappa^2}{\sqrt{n}}, \\
    \label{eq:devHS}
  \Psi_{\x,\lam} := \sqrt{\lam}\Psi_{\x} = \norm{\paren{\lam + B}^{-1/2}\paren{B - B_{\x}}}_{\hs}
    & \leq  L \sqrt\lam \bnl(\kappa,\kappa),\\
    \label{eq:devtheta}
    \Theta_{\z,\lam} :=
    \Theta_{\z} = \norm{\paren{\lam + B}^{-1/2}\paren{B_{\x} f_{\rho} - \incl_{\x}^{\ast}\y}} 
    & \leq  L \sqrt{\lambda}\bnl(M,\sigma),\\
    \label{eq:deveffdim}
     \abs{\NN(\lam) - \NN_{\x}(\lam)} & \leq L (1 + \sqrt{\NN_{\x}(\lam)})\bnl(\kappa,\kappa).
\end{align}
Inequality \eqref{eq:devHS} implies for any sublinear nondecreasing function
$\varphi$:
\begin{equation}
  \label{eq:phipertl}
  \Xi^\varphi_{\x,\lam} := \norm{\varphi\paren{ \lam + B}\varphi(\lam + B_{\x})^{-1}}   \leq
L^{2} \paren{\bnl(\kappa,\kappa) +1}^2,
\end{equation}
as well as the slightly sharper estimate
\begin{equation}
  \label{eq:cordespert}
  \Xi^r_{\x,\lam} := \norm{\paren{ \lam + B}^r(\lam + B_{\x})^{-r}}   \leq
L^{2r} \paren{\bnl(\kappa,\kappa) +1}^{2r}, \qquad r\in[0,1].
\end{equation}
Inequality \eqref{eq:deveffdim} implies provided that $\lambda\geq \frac{4\kappa^2}{n}$:
   \begin{equation}
       \label{eq:reldeveffdim}
       \max\paren{\frac{\NN(\lam)\vee 1}{\NN_{\x}(\lam)\vee 1}, \frac{\NN_{\x}(\lam)\vee 1}{\NN(\lam)\vee 1}}
       \leq \paren{1 +  \frac{4\kappa L}{\sqrt{\lambda n}}}^2.
   \end{equation} 
\end{prop}

The rest of this section is devoted to the proof of the above proposition and is
organized as follows. In Section~\ref{se:op_pert}, we first establish the detail of the (purely deterministic) argument
for the second  statement of the 
proposition, leading from~\eqref{eq:devHS} to~\eqref{eq:phipertl}-\eqref{eq:cordespert} 
via a perturbation argument. In Section~\ref{se:prob_bounds} following it,
we establish the main probabilistic estimates \eqref{eq:devHS}--\eqref{eq:deveffdim}
as well as~\eqref{eq:reldeveffdim}. 


\subsection{Operator perturbation bounds}
\label{se:op_pert}
The quantity~$\Psi_{\x}$ can be used to obtain the following purely
deterministic perturbation bounds which are crucial to our analysis of
generalized source conditions.  The main bound, related to Cordes'
Inequality, see~\cite[Thm.~IX.2.1-2]{Bat97},   will be given in
Proposition~\ref{prop:genperturb}, wich might be of independent
interest. 

We start with the following bound
using  a decomposition as introduced in \cite{Guo_2017}.
\begin{lem}
\label{lem:zhou}
  We have that
  \begin{equation}
    \label{eq:devfrid}
    \norm{\paren{I + S}\paren{I + S_{\x}}^{-1}-I}_{\hs} \leq \Psi_{\x} + \Psi_{\x}^{2}\,.
  \end{equation}
\end{lem}
\begin{proof}
  We use the decomposition from~\cite{Guo_2017} (eq. (29) there):
  \begin{multline}
    \label{eqzhou}
\paren{I + S}\paren{I + S_{\x}}^{-1} \\= I + \paren{I + S}^{-1}(S -
S_{\x}) 
+ (S - S_{\x})(I + S)^{-1}(S - S_{\x})(I + S_{\x})^{-1}.
    \end{multline}   
We bound
$$
\norm{\paren{I + S}^{-1}(S -S_{\x})}_{\hs}\leq \norm{\paren{I + S}^{-1/2}}
\norm{\paren{I + S}^{-1/2}(S -S_{\x})}_{\hs} \leq \Psi_{\x}.
$$
Also, we see
\begin{multline*}
\norm{ (S - S_{\x})(I + S)^{-1}(S - S_{\x})(I + S_{\x})^{-1}}_{\hs}
\\ \leq \norm{\paren{(I + S)^{-1/2}(S - S_{\x})}^{\ast}(I +
  S)^{-1/2}(S - S_{\x})}_{\hs}\norm{(I + S_{\x})^{-1}}\leq \Psi_{\x}^{2}.
\end{multline*}
Plugging the two last estimates together with the triangle inequality
in \eqref{eqzhou} gives the bound \eqref{eq:devfrid}.
\end{proof}
The next lemma allows to introduce an arbitrary nondecreasing sublinear function in the
perturbation estimate in HS-norm.
\begin{lem}
  \label{lem:phipert}
  Let $A,B$ be two self-adjoint positive invertible
operators on a separable Hilbert space.
Let $\varphi:\R_+ \rightarrow \R_+$ be a nondecreasing and sublinear function,
i.e., such that $\varphi(t)/t$ is nonincreasing. Then it holds
\begin{equation}
  \label{eq:phipert}
  \norm{\varphi(A) \varphi(B)^{-1} -I}_{\hs} \leq \norm{A B^{-1} - I}_{\hs}. 
  \end{equation}

\end{lem}
\begin{proof}
We start be establishing that for any sublinear function~$\varphi$ and any two positive numbers $\mu,\nu$:
\begin{equation}
  \label{eq:sublineq}
  \abs{\frac{\varphi(\mu)}{\varphi(\nu)} -1 } \leq
  \abs{\frac{\mu}{\nu} -1}\,.
  \end{equation}
  Indeed, if $\mu\geq \nu$, then $\varphi(\mu) \geq \varphi(\nu)$
  ($\varphi$ nondecreasing)
  but $\frac{\varphi(\mu)}{\varphi(\nu)} \leq \frac{\mu}{\nu}$ ($\varphi$ sublinear). Therefore
  \[
  \abs{\frac{\varphi(\mu)}{\varphi(\nu)} -1 }
  = \frac{\varphi(\mu)}{\varphi(\nu)} -1
  \leq \frac{\mu}{\nu}-1 = \abs{\frac{\mu}{\nu}-1}. 
  \]
  Similarly, if $\mu\leq \nu$, then $\varphi(\mu) \leq \varphi(\nu)$
  and $\frac{\varphi(\mu)}{\varphi(\nu)} \geq \frac{\mu}{\nu}$, so that
    \[
  \abs{\frac{\varphi(\mu)}{\varphi(\nu)} -1 }
  = 1 - \frac{\varphi(\mu)}{\varphi(\nu)}
  \leq 1- \frac{\mu}{\nu} = \abs{\frac{\mu}{\nu}-1}. 
  \]

  Now, let $(\nu_i,e_i)_{i\geq 1}$ be an eigendecomposition of $A$
  and $(\mu_i,f_i)_{i\geq 1}$  be an eigendecomposition of $B$. We have
  \begin{align*}
    \norm{AB^{-1} - I}_{\hs}^2
    & = \sum_{i\geq 1} \norm{ (AB^{-1} - I)f_i}^2\\
    & = \sum_{i\geq 1} \norm{ (\mu_i^{-1} A - I)f_i}^2\\
    & = \sum_{i\geq 1} \norm{ (\mu_i^{-1} A - I) \sum_{j\geq 1} \inner{f_i,e_j}e_j}^2\\
    & = \sum_{i,j\geq 1}  \paren{\frac{\nu_j}{\mu_i} -1}^2 \inner{f_i,e_j}^2.
  \end{align*}
  By the same token,
  \[
  \norm{\varphi(A)\varphi(B)^{-1} -I}_{\hs}^2 = \sum_{i,j\geq 1}  \paren{\frac{\varphi(\nu_j)}{\varphi(\mu_i)} - 1}^2 \inner{f_i,e_j}^2.
    \]
    Now apply inequality \eqref{eq:sublineq} to each term in the series to conclude.
\end{proof}
The main pertubation bound is given next.

\begin{prop}
  \label{prop:genperturb}
Let $\varphi:\R_+ \rightarrow \R_+$ be a nondecreasing and sublinear function, then it holds
\begin{equation}
  \label{eq:phimultpert}
  \norm{\varphi(I+S)\varphi(I+S_\x)^{-1}} \leq \paren{\Psi_\x +1}^2.
\end{equation}
In the case where $\varphi(t) = t^r$ for $r\in[0,1]$, we have the slightly sharper estimate
\begin{equation}
  \label{eq:cordes}
  \norm{(I+S)^r(I+S_\x)^{-r}} \leq \paren{\Psi_\x +1}^{2r}.
\end{equation}
\end{prop}
\begin{proof}
  Combining  \eqref{eq:devfrid} from Lemma~\ref{lem:zhou} and \eqref{eq:phipert}
  from Lemma~\ref{lem:phipert}, we obtain
\begin{align*}
\norm{\varphi(I+S)\varphi(I+S_\x)^{-1}}  & \leq 
\norm{\varphi(I+S)\varphi(I+S_\x)^{-1} - I } +1 \\
& \leq \norm{\varphi(I+S)\varphi(I+S_\x)^{-1} - I } _{\hs}+1 \\
& \leq \norm{(I+S)(I+S_\x)^{-1} - I } _{\hs}+1 \\
& \leq \Psi_\x + \Psi_\x^2 +1 \leq (\Psi_\x +1)^2,
\end{align*}
which is the announced inequality~\eqref{eq:phimultpert}. The second estimate is
a consequence of the first with $\varphi(x)=x$ followed by the Cordes' inequality
(Theorem IX.2.1-2. in \cite{Bat97}).
\end{proof}

This yields the statement leading from~\eqref{eq:devHS} to~\eqref{eq:phipertl}, by applying~\eqref{eq:phimultpert}
to $\wt{\varphi}(u):=\varphi(\sqrt{\lambda}u)$, which is sublinear noncreasing since $\varphi$ is; and using $L\geq 1$.

\subsection{Probabilistic bounds}
\label{se:prob_bounds}
\pmathe{We turn to the probabilistic bounds, and as in previous
references \cite{BlaMuc18,CaponnettoDevito2007,CaponnettoYao2010} we
will apply a Hoeffding- or Bernstein-type deviation inequality for
Hilbert-valued random variables, see~\cite{PinSak86}. }

\begin{lem}\label{lem:psix-bound}
Under Assumption~\ref{ass:bernstein}, for $\eta \in (0,1)$, each one of the inequalities
\begin{align}
\label{ali:gammax} \Gamma_\x & \leq 2 \log(2/\eta) \lambda^{-1} \frac{\kappa^2}{\sqrt{n}}; \\
\label{ali:psix}  \Psi_{\x} & \leq 2
  \log(2/\eta) \bnl(\kappa,\kappa),
\end{align}
holds with probability~$1 - \eta$ over the draw of $\x=(x_1,\ldots,x_n) \stackrel{i.i.d.}{\sim} \rho_X $.
\end{lem}
\begin{proof}  
  Introducing the shorthand $S_x := \bar K_x \otimes \bar K_x^*$
  we have $S = \E_{X \sim \rho_X}[S_X]$ and $\norm{S_x}_{\hs} \leq \lambda^{-1} \kappa^2$.
  By the Hoeffding-type inequality in Hilbert space (see e.g.~\cite{Pin91}),
  we obtain that \eqref{ali:gammax} holds with probability~$1 - \eta$.
  Now, we consider the function
  \begin{equation}
    \label{eq:xi}
    \xi(x) := \paren{I + S}^{-1/2}S_x,\quad x \in X,
  \end{equation}
  so that  $(I+S)^{-1/2}S = \E_{X\sim \rho_X}[\xi(X)]$, and
  $(I+S)^{-1/2}S_\x = n^{-1} \sum_{i=1}^n \xi(x_i)$\,.
  
  We see that~$\norm{\xi(x)}_{\hs} \leq \norm{S_x}_{\hs}=\norm{\bar K_x \otimes \bar K_x^*}_{\hs} \leq \lambda^{-1}
  \kappa^2 =\colon \frac C 2 $.
Also, we find
\begin{align*}
\E_{X \sim \rho_X}\norm{\xi(X)}_{\hs}^{2}
&= \E_{X \sim \rho_X}\brac{\tr{S_{X}\paren{I + S}^{-1}S_{X}}}\\
& \leq \lambda^{-1} \kappa^2 \,\E_{X \sim \rho_X}\brac{\tr{\paren{I + S}^{-1}S_{X}}}\\
& = \lambda^{-1} \kappa^2 \NN =: s^{2}.
\end{align*}
Therefore, by the Bernstein-type inequality in Hilbert space (see e.g.~\cite[Prop.~A.1]{BlaMuc18}),
we obtain that with probability~$1 - \eta$ we have
\begin{align*}
\Psi_{\x} \leq \paren{2 \log(2/\eta)} \paren{\frac C n + \frac s
  {\sqrt{n}}}
& =  2 \log(2/\eta)
\paren{ {\frac{2\kappa^2}{\lambda n}} + \sqrt{\frac{\kappa^2\NN}{\lambda n}}}\\
& =  2 \log(2/\eta) \bnl(\kappa,\kappa).
\end{align*}
which gives~\eqref{ali:psix} and completes the proof.
\end{proof}  
We turn to derive a bound similar
to~\cite[Prop.~5.2]{BlaMuc18}.  We
have that
$$
(B + \lambda)^{-1/2}(B_{\x} \frho - \incl_{\x}^{\ast}y) = \frac 1 n \sum_{j=1}^{n} \brac{\paren{I +
  S}^{-1/2} (y_{j} - \frho(x_{j}))\bar K_{x_{j}}}.
$$
\begin{lem}
  \label{lem:thetazbound}
  Under Assumption~\ref{ass:bernstein},
  for $\eta \in (0,1)$, with probability~$1 - \eta$ over the draw of $\z=((x_1,y_1),\ldots,(x_n,y_n)) \stackrel{i.i.d.}{\sim} \rho $, it holds
$$
\Theta_{\z}:= \norm{\frac 1 n \sum_{j=1}^{n} \brac{\paren{I +
      S}^{-1/2} (y_{j} - \frho(x_{j}))\bar K_{x_{j}}}} \leq
\paren{2 \log(2/\eta)} \sqrt{\lambda} \bnl(M,\sigma).
$$
\end{lem}
\begin{proof}
  We start, as in~\cite{BlaMuc18} with the observation that
  \begin{equation}
    \label{eq:kxsn}
\E_{X\sim\rho_X} \norm{ (I + S)^{-1/2}\bar K_{X}}_{\Hr}^{2} = \NN.
\end{equation}
We introduce the function
\begin{equation}
  \label{eq:ran-elem}
  \xi(x,y) := \paren{I +  S}^{-1/2} (y - \frho(x))\bar K_{x}, \quad
  x \in \X, y \in \R,
\end{equation}
so that $\Theta_{\z} = \norm{n^{-1}\sum_{i=1}^n \xi(x_i,y_i)}$
and $\E_{(X,Y)\sim \rho}[ \xi(X,Y)]=0$.
Then we bound the $m$-th moments as in~\cite[Proof of
Prop.~5.2]{BlaMuc18} using \eqref{eq:kxsn}
and the noise moments assumption, to obtain
$$
\E_{(X,Y) \sim \rho}\norm{\xi(X,Y)}^{m} \leq \frac 1 2 m!\paren{\sigma \sqrt{\NN}}^{2}
\paren{M\kappa \lambda^{-1/2}}^{m-2}.
$$
Again, an application of the Bernstein deviation inequality in Hilbert space \cite[Prop.~A.1]{BlaMuc18} completes the proof.
\end{proof}

Finally, we derive the following relative deviation bound between the effective
dimension and its empirical counterpart, see~(\ref{it:effdim})
and~(\ref{it:effdimemp}), respectively. A result of a similar flavor
can be found in \cite[Prop.~1]{RudCamRosNIPS2015}.  The version below is somewhat streamlined with simpler assumptions, a slightly
different proof technique (based again on a decomposition of the type~\eqref{eqzhou}, inspired from~\cite{Guo_2017}), giving that the first estimate holds
without any restrictions on $n$ nor $\lambda$. Recall the assumption
that~$\kappa^2 = \sup_{x\in \X} K(x,x) < \infty$.
\begin{lem}\label{lem:n-nx}
  Let $\lam>0$ be fixed. 
For $\eta \in (0,1/2)$,  with probability~$1- 2\eta$ we have that
   \begin{equation}
     \label{eq:n-nx}
     \abs{\NN - \NN_{\x}} \leq 2 \log(2/\eta) (1 + \sqrt{\NN_{\x}})\bnl(\kappa,\kappa),
   \end{equation}
   which implies, provided $\lambda\geq \frac{4\kappa^2}{n}$:
   \begin{equation}
       \label{eq:n-nx-relative}
       \max\paren{\frac{\NN\vee 1}{\NN_{\x}\vee 1}, \frac{\NN_{\x}\vee 1}{\NN\vee 1}}
       \leq \paren{1 +  {\frac{8\kappa \log\paren{2 \eta^{-1}}}{\sqrt{\lambda n}}}}^2
       \leq 25 (\log(2\eta^{-1}))^2.
   \end{equation}
\end{lem}
\begin{proof}
  We verify that
$$
(I + S)^{-1}S - (I+S_{\x})^{-1}S_{\x} = (I+S)^{-1}(S - S_{\x}) +
(I+S)^{-1}(S - S_{\x})(I + S_{\x})^{-1}S_{\x}.
$$
Letting
\begin{align*}
  I_{1} &:= \abs{\tr{(I+S)^{-1}(S - S_{\x})}},\intertext{and}
          I_{2}&:= \abs{\tr{(I+S)^{-1}(S - S_{\x})(I + S_{\x})^{-1}S_{\x}}},
\end{align*}
we see that
$$
\abs{\NN- \NN_{\x}} \leq I_{1} + I_{2}.
$$
We introduce the non-negative real random variables~$\xi_i:= \tr{(I+S)^{-1}
  S_{x_i}}\geq 0$, (recalling the shorthand $S_x := \bar K_x \otimes \bar K_x^*$). Clearly, the $\xi_i$ are i.i.d. random variables with $\E[\xi_i] = \NN$,
and $\frac{1}{n}\sum_i \xi_i = \tr{(I+S)^{-1}S_{\x}}$. We shall bound the
deviation using the classical (real valued) Bernstein's inequality.
Denoting $\xi$ a random variable distributed as $\xi_1$ for short,
we find that~$\xi \leq \tr{S_{X}}\leq
\kappa^2/\lambda $. Also, since $\xi \geq 0$ we bound
$$
\E[\xi^{2}]\leq \lambda^{-1} \kappa^2 \E[\xi] = \lambda^{-1} \kappa^2 \NN.
$$
Bernstein's inequality  gives with probability~$1-\eta$ that
$$
I_{1} = \abs{\E[\xi] - \frac{1}{n}\sum_i\xi_i} \leq 2 \log(2/\eta)\paren{\frac{2 \kappa^2}{\lambda n} +
  \sqrt{\frac{\kappa^2 \NN}{\lambda  n}}} =  2 \log(2/\eta) 
\bnl(\kappa,\kappa).
$$
For bounding~$I_{2}$ we argue as follows.
\begin{multline}  \label{eq:product}
I_{2} = \abs{
\scalar{(I+S)^{-1}(S - S_{\x})}
{(I +S_{\x})^{-1}S_{\x}}{\hs}
  }\\
\leq  \norm{(I+S)^{-1}(S - S_{\x})}_{\hs}
 \norm{(I + S_{\x})^{-1}S_{\x}}_{\hs},
\end{multline}
and we bound each factor.The first one is bounded as 
\begin{equation}
  \label{eq:1bound}
  \norm{(I+S)^{-1}(S - S_{\x})}_{\hs} \leq\norm{(I+S)^{-1/2}(S -
    S_{\x})}_{\hs} = \Psi_{\x}.
\end{equation}
For the second one, we observe that
$$
\norm{(I + S_{\x})^{-1}S_{\x}}_{\hs}^{2} \leq \norm{(I +
  S_{\x})^{-1}S_{\x}}_{\Hr\to \Hr} \norm{(I +
  S_{\x})^{-1}S_{\x}}_{1}\leq \NN_{\x}.
$$
Using Lemma~\ref{lem:psix-bound}, overall
we obtain that with probability~$1 - 2\eta$ it holds
$$
I_{1} + I_{2} \leq  2 \log(2/\eta)\bnl(\kappa,\kappa) + \Psi_{\x} \sqrt{\NN_{\x}}
\leq 2 \log(2/\eta) (1 + \sqrt{\NN_{\x}})\bnl
(\kappa,\kappa),$$
which completes the proof of~\eqref{eq:n-nx}.
%
%

Put $A:= \sqrt{\NN(\lam)}$\,, $B := \sqrt{\NN_{\x}(\lam)}$\,, $\delta:= \frac{2\kappa}{\sqrt{\lambda n}}$\,, and $L=\log(2/\eta)$,
then one can rewrite~\eqref{eq:n-nx} as $\abs{A^2 - B^2} \leq L \delta (1+B)(\delta + A)$ (holding with probability $1-2\eta$).

If we assume from now on that~$\lambda \geq \frac{4\kappa^2}{n}$, then it holds $\delta \leq 1$;
putting $\bar{A}:=A\vee 1$, $\bar{B}:=B \vee 1$, we therefore have with high probability
$|\bar{A}^2 - \bar{B}^2| \leq \abs{A^2-B^2} \leq 4 L\delta \bar{A} \bar{B}$,
entailing
\[
  \max\paren{\frac{\bar{A}}{\bar{B}},\frac{\bar{B}}{\bar{A}}} \leq 1 +
  \abs{\frac{\bar{A}}{\bar{B}}-\frac{\bar{B}}{\bar{A}}} \leq 1 + 4L\delta.
\]
\end{proof}

To wrap up the proof of~\eqref{eq:devHS}--\eqref{eq:deveffdim}
in Proposition~\ref{prop:probestimates}, we collect
Lemmas~\ref{lem:psix-bound},~\ref{lem:thetazbound} and~\ref{lem:n-nx}
together with a union bound on the different events involved. Note the insignificant
technical point that Lemma~\ref{lem:psix-bound} is used 
in the proof of Lemma~\ref{lem:n-nx} already, so that we don't have
to pay again in the union bound for the event appearing in~\eqref{eq:devHS},
it has been already counted in~Lemma~\ref{lem:n-nx}. Overall for fixed $\lambda>0$
all estimates taken together required to use three Bernstein's inequalities
and one Hoeffding's inequality,
hence replacing $\eta$ by $\eta/4$ in the individual inequalities stated in the three lemmas
to obtain Proposition~\ref{prop:probestimates}.


 
\section{Proof of Proposition~\ref{prop:probestimates_main2}}
\label{sec:proof-prop2}

\pmathe{The first two statements in Proposition~\ref{prop:probestimates_main2}
are direct consequences of
Proposition~\ref{prop:probestimates}. Specifically,~\eqref{eq:phipertl} gives a bound for
$\norm{\paren{ \lam + B}(\lam + B_{\x})^{-1}}$ by choosing $\varphi$
as the identity function.
  Also, estimate~\eqref{eq:cordespert}
(with the choice $r=\frac{1}{2}$) and~\eqref{eq:reldeveffdim} directly
establish~\eqref{eq:phipertl2} and~\eqref{eq:reldeveffdim2} of
Proposition~\ref{prop:probestimates_main2}.
It remains to establish the estimate~(\ref{eq:assmpt1prob}).
}

We start with mentioning the following elementary property of sublinear functions.
\begin{lem}
  \label{le:subadd}
A sublinear index function~$\phi$ is subadditive, i.e., 
 $$
\phi(s+t) \leq \phi(s) + \phi(t),\quad s,t>0. 
$$
\end{lem}
\begin{proof}
Sublinearity of~$\phi$ yields~$s\phi(s+t) \leq (s+t)\phi(s)$, and also~$t\phi(s+t) \leq (s+t)\phi(t)$. Summing both inequalities allows to complete the proof. 
\end{proof}

\begin{proof}[Proof of estimate~\eqref{eq:assmpt1prob}]
  We use arguments similar to those appearing in
  \cite{BlaMuc18,Rastogi2017,LuMaPer2018,CaponnettoYao2010}.
  The treatment of generalized source conditions has
  been considered
  in \cite{Rastogi2017} and we use arguments very close in spirit to that reference,
  deriving here inequalities that are tailored for our needs.

Since in several estimates below, the operator~$\paren{B_{\x} +
  \lam}^{1/2}$ occurs, we refer to the following general identity,
see~\cite[(A.33)]{LuMaPer2018}:
For any continuous (measurable) function~$m\colon
[0,\kappa^{2}]\to \R$ we find that
\begin{multline}
  \label{eq:f1/2-bound}
  \norm{m(B_{\x})\paren{\lambda I + B_{\x}}^{1/2}}   =
  \norm{m(B_{\x})^2(\lambda I +B_{\x})}^{\frac{1}{2}} \\
  \leq \paren{\lambda \norm{m(B_{\x})^2} +
    \norm{m(B_{\x})^2B_{\x}}}^{\frac{1}{2}} 
  =  \paren{\lambda \norm{m(B_{\x})}^2 + \norm{m(B_{\x})B_{\x}^{1/2}}^2}^{\frac{1}{2}}.  
\end{multline}

In the rest of this proof, we will use the notation defined in
Proposition~\ref{prop:probestimates}.
Below we will use several times the regularization and qualification estimates from
Proposition~\ref{prop:interpreg} and write $\ol{\gamma}$ as a shorthand for $\ol{\gamma}_\psi$.

We start with the following bound.
  \begin{multline*}
    \norm[1]{    \paren{B_{\x} + \lam}^{1/2}\paren{\fo - \fzl}}_{\Hr}\\
    \begin{aligned}
    & =  \norm[1]{    \paren{B_{\x} + \lam}^{1/2}\paren{\fo - g_{\lam}(B_{\x})\incl^*_{\x}\y}}_{\Hr}\\
    & \leq  \norm[1]{    \paren{B_{\x} + \lam}^{1/2}\paren{\fo - g_{\lam}(B_{\x})\incl^*_{\x} \incl_{\x} \fo}}_\Hr + 
      \norm[1]{\paren{B_{\x} + \lam}^{1/2}g_{\lam}(B_{\x})(\incl^*_{\x}(\y-\incl_{\x}\fo))}_{\Hr}\\
    & =: T_1 + 
    T_2.
  \end{aligned}
\end{multline*}

For the second term, we have, using~\eqref{eq:supg},~\eqref{eq:supgt}:
\begin{align*}
  T_2
  & = \norm[1]{\paren{B_{\x} + \lam}^{1/2}g_{\lam}(B_{\x})(B_{\x} \fo - \incl_{\x}^{\ast}y)}_\Hr\\
  & \leq \norm{g_{\lam}(B_{\x})(B_\x +\lambda)}
    \norm[1]{(B_\x +\lambda)^{-\frac{1}{2}} (B +\lambda)^{\frac{1}{2}}}
    \norm[1]{(B +\lambda)^{-\frac{1}{2}}(B_{\x} \fo - \incl_{\x}^{\ast}y)}_\Hr\\
  & \leq 
    2\ol{\gamma}
    \Xixl^{1/2}
    \Tzl. 
\end{align*}
For the first term,
\begin{align*}
    T_1 =  \norm{    \paren{B_{\x} + \lam}^{1/2}r_{\lam}(B_{\x}) \fo}_\Hr 
           &  \leq   \norm{    \paren{B_{\x} + \lam}^{1/2}r_{\lam}(B_{\x}) \phi(B)}\\
    & = \norm{    \paren{B_{\x} + \lam}^{1/2}r_{\lam}(B_{\x}) \varphi_1(B) \varphi_2(B)} \\
    & \leq \norm{    \paren{B_{\x} + \lam}^{1/2}r_{\lam}(B_{\x})
      \varphi_1(B_{\x}) \varphi_2(B)} \\
    & \qquad \qquad + \norm{    \paren{B_{\x} + \lam}^{1/2}r_{\lam}(B_{\x})
      (\varphi_1(B) - \varphi_1(B_{\x})) \varphi_2(B)} \\ 
    & := T_3 + T_4.
\end{align*}
We bound the terms in turn:
\begin{align*}
  T_4 & = \norm{    \paren{B_{\x} + \lam}^{1/2}r_{\lam}(B_{\x}) (\varphi_1(B) - \varphi_1(B_{\x})) \varphi_2(B)}\\
  & \leq \norm{    \paren{B_{\x} + \lam}^{1/2}r_{\lam}(B_{\x})}
    \lipc \norm{B - B_{\x}}_{\hs} \varphi_2(\kappa^2)\\
  & \leq 
\ol{\gamma} \lipc \varphi_2(\kappa^2) 
\sqrt{\lambda} \Gxl;
\end{align*}
above, we applied~(\ref{eq:f1/2-bound}) with~$m:= r_{\lam} $ for the second inequality;
for the first inequality we used the well-known fact that
since $\varphi_1$ is $\ell$-Lipschitz as a function of real variable, it is also
$\ell$-Lipschitz for the HS-norm when acting on self-adjoint Hilbert-Schmidt operators.
Finally, {using that the function~$\varphi_{2}$ is sublinear hence subadditive (see Lemma~\ref{le:subadd}) and nondecreasing,} we estimate
\begin{align*}
  T_3 & = \norm{    \paren{B_{\x} + \lam}^{1/2}r_{\lam}(B_{\x}) \varphi_1(B_{\x}) \varphi_2(B)}\\
  & \leq \norm{    \paren{B_{\x} + \lam}^{1/2}r_{\lam}(B_{\x}) \varphi_1(B_{\x}) \varphi_2(B_{\x} + \lambda)}
  \norm{\varphi_2(B_\x +\lambda)^{-1}\varphi_2(B+\lambda)} \\
  & \leq \paren{\norm{    \paren{B_{\x} + \lam}^{1/2}r_{\lam}(B_{\x}) \phi(B_{\x})}
    + \norm{  \paren{B_{\x} + \lam}^{1/2}r_{\lam}(B_{\x})
      \varphi_1(B_{\x})} \varphi_2(\lambda)} \Xixl^{\varphi_{2}}.
\end{align*}
Using the identity~(\ref{eq:f1/2-bound}) in both summands above, we
bound
$$
\norm{\paren{B_{\x} + \lam}^{1/2}r_{\lam}(B_{\x}) \phi(B_{\x})}
\leq 
2\ol{\gamma} \sqrt\lam \phi(\lam),
$$
and
$$
\norm{  \paren{B_{\x} + \lam}^{1/2}r_{\lam}(B_{\x}) \varphi_1(B_{\x})}
\leq 
2\ol{\gamma} \sqrt\lam \varphi_{1}(\lam).
$$
Overall we obtained that
$$
 T_3 \leq 
4 \ol{\gamma} \Xixl^{\varphi_{2}} \sqrt\lam \phi(\lam).
$$
In these terms we have obtained the bound
\begin{equation}
  \label{eq:compound-bound1}
  \norm{\paren{B_{\x} + \lam}^{1/2}\paren{\fo - \fzl}}_{\Hr}
   \leq 
    4 \ol{\gamma} \paren{\lipc \varphi_2(\kappa^2) \sqrt\lam \Gxl + \sqrt\lam \phi(\lam) \Xixl^{\varphi_{2}}
    + 
    \Tzl \Xixl^{1/2}},
\end{equation}
provided that the chosen regularization has qualification~$t \mapsto \sqrt t \phi(t)$.
Now assume the event $\Omega_{\lambda,\eta}$ of probability at least
$1-\eta$ from Proposition~\ref{prop:probestimates} is satisfied; we
plug in
estimates~\eqref{eq:devHShoeff},~\eqref{eq:phipertl},~\eqref{eq:cordespert}
and~\eqref{eq:devtheta} to obtain~\eqref{eq:assmpt1prob}.
\end{proof}
This completes the proof of Proposition~\ref{prop:probestimates_main2}.


\noindent
\\
{\large \bf Acknowledgments}
\\ 
\vspace*{1ex}
NM is supported by the German Research Foundation under  DFG Grant STE 1074/4-1.


\bibliography{bibliography}
\bibliographystyle{plain}

\end{document}